\documentclass[11pt,a4paper]{article}
\usepackage[T1]{fontenc}
\usepackage[utf8]{inputenc}
\usepackage{lmodern}
\usepackage{amsmath,amssymb,amsthm,mathtools}
\usepackage{relsize}
\usepackage{pifont}
\usepackage[bbgreekl]{mathbbol}
\DeclareSymbolFontAlphabet{\mathbb}{AMSb}
\DeclareSymbolFontAlphabet{\mathbbl}{bbold}
\usepackage[margin=1.05in]{geometry}
\usepackage{microtype}
\usepackage{indentfirst}
\usepackage{tikz}
\usepackage{enumitem}
\usepackage[hidelinks,unicode]{hyperref}
\setlist[enumerate]{leftmargin=1.8em,itemsep=3pt,topsep=3pt}
\setlist[itemize]{leftmargin=1.5em,itemsep=3pt,topsep=3pt}
\newcommand{\Prism}{{\mathlarger{\mathbbl{\Delta}}}}
\newcommand{\X}{\mathfrak X}
\newcommand{\Oc}{\mathcal O_C}
\newcommand{\Ainf}{A_{\mathrm{inf}}}
\newcommand{\BdR}{B_{\mathrm{dR}}^+}
\newcommand{\RG}{R\Gamma}
\newcommand{\otL}{\mathbin{\otimes^{\mathrm L}}}
\newcommand{\cofib}{\operatorname{cofib}}

\newcommand{\Tot}{\operatorname{Tot}}
\newcommand{\rk}{\operatorname{rk}}
\newcommand{\im}{\operatorname{im}}
\newcommand{\coker}{\operatorname{coker}}

\newcommand{\gr}{\operatorname{gr}}
\newcommand{\cris}{\mathrm{cris}}
\newcommand{\HW}{\mathrm{HW}}
\newcommand{\Hdg}{\mathrm{Hdg}}
\newcommand{\Nyg}{\mathcal N}
\newcommand{\calQ}{\mathcal Q}
\newcommand{\calP}{\mathcal P}

\newcommand{\HH}{\mathrm{H}}
\newcommand{\dc}{%
  \mathrel{\mathchoice
    {\dcl{0.16}}
    {\dcl{0.16}}
    {\dcl{0.11}}
    {\dcl{0.08}}
  }%
}
\newcommand{\dcl}[1]{\,%
  \mathrel{\raisebox{-0.15ex}{\hbox{%
  \begin{tikzpicture}[scale=#1, line width=0.6pt]
    \draw (0,0) -- (2,0) -- (1.2,1.75) -- (0.8,1.75) -- cycle;
    \draw (1.6,0)-- (0.8,1.75);
  \end{tikzpicture}}}%
}}
\newcommand{\Ht}{\widetilde{\mathrm{H}}}
\newcommand{\taut}{\widetilde\tau}

\newcommand{\Frac}{\operatorname{Frac}}

\newcommand{\cdstimes}{{\scriptstyle
  \widehat{\text{\ding{84}}^{{\scriptscriptstyle\mathrm{L}}}}}}

\theoremstyle{plain}
\newtheorem{theorem}{Theorem}[section]
\newtheorem{proposition}[theorem]{Proposition}
\newtheorem{lemma}[theorem]{Lemma}
\newtheorem{corollary}[theorem]{Corollary}
\newtheorem{conjecture}[theorem]{Conjecture}
\theoremstyle{definition}

\newtheorem{remark}[theorem]{Remark}
\newtheorem{notation}[theorem]{Notation}

\numberwithin{equation}{section}
\hypersetup{
 pdftitle={Hodge and Hodge-Witt polygons under specialization},
 pdfsubject={Mixed-characteristic comparison, equal-characteristic specialization, and strict
 examples},
 pdfauthor={Yuan Yang},
 pdfkeywords={Hodge-Witt, Nygaard filtration, prismatic cohomology, polygon comparison}
}
\title{Hodge and Hodge--Witt polygons under specialization}
\author{Yuan Yang}
\date{\today}
\begin{document}
\maketitle
\begin{abstract}
We show that, in equal characteristic $p$, the Hodge--Witt polygon can only move down and
domino numbers can only increase under specialization in smooth proper families. We also show that the Hodge polygon of the
generic fiber of a proper smooth $p$-adic formal scheme over $\Oc$ lies on or above the Hodge--Witt
polygon of its special fiber, where $C$ is a complete algebraically closed extension of
$\mathbb{Q}_p$. We further show that, in this mixed-characteristic setting, if the crystalline cohomology of the
special fiber is torsion-free in degrees $n$ and $n+1$, then the degree-$n$ Hodge polygon of the
generic fiber lies on or below the abstract Hodge polygon of the special fiber. We illustrate these comparisons with explicit examples. 
\end{abstract}

\tableofcontents

\section{Introduction}\label{sec:statement}
The study of Newton polygons under specialization and of their comparison with Hodge polygons goes
back to conjectures of Grothendieck and Katz, respectively, and to work of Katz, Mazur, Ogus
\cite{Mazur,BO}, and Crew \cite{CrewSpecialization}. Katz and Crew proved that Newton polygons can only move up under specialization
\cite[Corollary~2.3.2]{KatzSlope}, \cite[Theorem~2.1]{CrewSpecialization}, and it is a basic fact in $p$-adic Hodge theory that 
the Hodge polygon of the generic fiber lies on or below the Newton polygon of the special fiber. 
We prove three further polygon comparisons, mainly concerning the
Hodge--Witt polygon of Illusie and Ekedahl \cite{Ekedahl}, beginning in equal characteristic \(p\) and
then passing to mixed characteristic.

For a perfect field $k$ of characteristic $p$, write $W=W(k)$. In mixed characteristic,
$k$ denotes the residue field of $C$. In cohomological degree $n$, our convention is that the Hodge polygon, denoted \(\Hdg_n(X)\), has slope $i$ with multiplicity
$h^{i,n-i}=\dim \mathrm{H}^{n-i}(X,\Omega_X^i)$, where the dimension is taken over the base field.
When all Hodge--Witt numbers are nonnegative, the Hodge--Witt polygon, denoted \(\HW_n(X)\), similarly has slope $i$ with
multiplicity $h_W^{i,n-i}$. In general, we use the halfspace definition, 
which also allows negative Hodge--Witt numbers. The Newton polygon is denoted \(\operatorname{Newt}_n(X)\).

Our first comparison shows that the loci defined by bounding the Hodge--Witt polygons from above in a smooth proper
family in characteristic $p$ are closed.

\begin{theorem}\label{thm:equalchar}
Let \(S\) be a connected \(\mathbb F_p\)-scheme. Let $f:X\to S$ be a proper smooth morphism of
$\mathbb F_p$-schemes with connected geometric fibers. Fix an integer $n\geqslant0$.
For $x\in S$, write
$X_{\bar x}$ for the geometric fiber over $x$, and let \(b\) be the common $n$-th Betti number of \(X_{\bar x}\). 
For any continuous function $P:[0,b]\to\mathbb R$, the locus
\[\left\{x\in S:\HW_n(X_{\bar x})\leqslant P\right\}\]
is Zariski closed. Moreover, it is locally the zero set of a finitely generated ideal. 
\end{theorem}

In particular, Hodge--Witt polygons can only move down under specialization: if $y$ is a specialization of $x$,
then \(\HW_n(X_{\bar y})\leqslant\HW_n(X_{\bar x})\).
This direction of semicontinuity is opposite to that of Newton polygons in \cite[Corollary~2.3.2]{KatzSlope} \cite[Theorem~2.1]{CrewSpecialization}. 
Together with Newton polygon specialization, this gives, for any specialization $x\leadsto y$,
a chain of polygons
\[
 \operatorname{Newt}(X_{\bar y})\ \geqslant\ \operatorname{Newt}(X_{\bar x})
 \ \geqslant\ \HW(X_{\bar x})\ \geqslant\ \HW(X_{\bar y})
\]
in each cohomological degree.

As a corollary, for any proper smooth morphism $f:X\to S$ as above, any $i,j\geqslant0$,
and any $m\in\mathbb Z$, the locus where the domino number \(T^{i,j}\) is at least $m$
\[
 \left\{x\in S:T^{i,j}(X_{\bar x})\geqslant m\right\}
\]
is also Zariski closed and locally the zero set of a finitely generated ideal (Corollary~\ref{cor:dominospecialization}). 
Thus every domino number is upper
semicontinuous and can only increase under specialization. This extends Crew's semicontinuity
theorem for $T^{0,2}$ of surfaces \cite[Theorem~7]{Crew} to all domino numbers in arbitrary
dimension. The Hodge--Witt locus is consequently open (Corollary~\ref{cor:hwopen}).

The decrease can be strict for threefolds in every characteristic $p>0$. Dimension \(3\) is minimal:
as Ekedahl observed \cite[p.~114]{Ekedahl}, Hodge--Witt numbers are locally constant in proper
smooth families of surfaces. Moreover, whenever a surface lifts to characteristic zero, the
Hodge--Witt numbers of the special fiber equal the Hodge numbers of the lifting. 

Our second comparison relates the Hodge polygon of a characteristic-zero generic fiber to the
Hodge--Witt polygon of its special fiber.

\begin{theorem}\label{thm:main}
Let $C$ be a complete algebraically closed non-archimedean extension of $\mathbb{Q}_p$, with ring of
integers $\Oc$ and residue field $k$. Let $\X$ be a proper smooth $p$-adic formal scheme over 
$\Oc$ with connected geometric fibers. Write $X_\eta$ for its rigid analytic generic
fiber, $X_k=\X\otimes_{\Oc}k$ for its special fiber, and $W=W(k)$. Fix $n\geqslant0$. 
Then \[\Hdg_n(X_\eta)\ \geqslant\ \HW_n(X_k),\]and the two polygons have the same endpoints.
\end{theorem}

Thus the usual Newton--Hodge comparison extends to
\[
 \operatorname{Newt}(X_k)\ \geqslant\ \Hdg(X_\eta)
 \ \geqslant\ \HW(X_k).
\]
The theorem requires neither projectivity, algebraization, nor descent to a discretely valued field.
Crystalline torsion is allowed, and no degeneration hypothesis is imposed on the special-fiber
Hodge--de~Rham spectral sequence. 

The generic-fiber Hodge polygon is not, in general, determined by the special fiber. Li constructed
smooth projective threefolds with isomorphic special fibers whose characteristic-zero liftings have
different Hodge numbers \cite[Proposition~2.3]{Li20}. Theorem~\ref{thm:main} nevertheless provides a
common lower bound for the Hodge polygons of all proper smooth liftings covered by the theorem. 
These examples are analyzed in Section~\ref{subsubsec:li}.

Our third comparison supplies a complementary upper bound. The abstract Hodge polygon
$\Hdg^{\mathrm{abs}}(X_k)$ is defined by the elementary divisors of crystalline Frobenius on the
torsion-free quotient of crystalline cohomology; see Section~\ref{subsec:abstracthodge}.

\begin{theorem}\label{thm:abstracthodge}
In the setting of Theorem~\ref{thm:main}, fix $n\geqslant0$, and suppose that
$\mathrm{H}^n_{\cris}(X_k/W)$ and $\mathrm{H}^{n+1}_{\cris}(X_k/W)$ are torsion-free. Then
\[\Hdg^{\mathrm{abs}}_n(X_k)\ \geqslant\ \Hdg_n(X_\eta)\]
with common endpoints. In particular,
\[
 \operatorname{Newt}_n(X_k)\ \geqslant\ \Hdg^{\mathrm{abs}}_n(X_k)
 \ \geqslant\ \Hdg_n(X_\eta)\ \geqslant\ \HW_n(X_k)
\]
and all four polygons have the same endpoints.
\end{theorem}

The abstract Hodge numbers of \(X_k\) in degree \(n\) equal
both the Hodge numbers and the Hodge--Witt numbers of the canonical Mazur--Ogus subquotient 
of the diagonal cohomology of \(X_k\)
\cite[I, Theorem~4.5; IV, proof of Theorem~3.5]{Ekedahl}. 
Under the above torsion-freeness hypotheses, if the degree-$n$ diagonal cohomology of $X_k$ is already
Mazur--Ogus in the sense of \cite[I, Definition~4.3]{Ekedahl}, the two bounds coincide.
Consequently, the degree-$n$ Hodge numbers of every
proper smooth lifting are determined by $X_k$: they equal $h_W^{i,n-i}(X_k)$
(Corollary~\ref{cor:determinedhodge}).

We do not know whether the crystalline torsion-freeness hypotheses in
Theorem~\ref{thm:abstracthodge} can be removed. At the algebraic level, there are nonfree
Breuil--Kisin--Fargues modules whose crystalline specialization has an abstract Hodge polygon
lying strictly below their generic Hodge polygon (Remark~\ref{rem:nonfreebkf}). However, we have
neither a geometric counterexample nor a proof of the comparison without these hypotheses.

Our strategy is to compare the degree-$n$ cohomology lengths of quotients of the prismatic complex
by Andriopoulos's $r$-Nygaard filtration \cite[\S5]{Andriopoulos} and to study their asymptotic
behavior as $r$ tends to infinity. In characteristic $p$, we work directly with the quotient
$W\Omega^\bullet/\Nyg^{\geqslant a}_rW\Omega^\bullet$, where \(\Nyg^{\geqslant a}_rW\Omega^\bullet\)
corresponds exactly to Andriopoulos's $r$-Nygaard filtration in \cite[\S5]{Andriopoulos}. 
In mixed characteristic, we use \(\varphi^*(\Prism)/\Nyg^{\geqslant a}_r\varphi^*(\Prism)\), where 
\(\Prism=\Prism_{\mathfrak{X}/\Ainf}\coloneqq R\nu_* \mathcal{O}_{\Prism}\in D(\mathfrak{X}_{\mathrm{\acute{e}t}}, \Ainf)\)
is the prismatic complex in \cite[Construction 4.9]{BS} and \(\Nyg^{\geqslant a}_r\) is 
Andriopoulos's $r$-Nygaard construction in \cite[Theorem~1.1]{Andriopoulos}. 
In characteristic \(p\), the Hodge--Witt numbers govern the asymptotic growth of the corresponding
\(W(k)\)-cohomology lengths as \(r\) tends to infinity. In mixed characteristic, for
\(0\!\leqslant\!j\!\leqslant\!r\!-\!1\), the \(\varphi^{-j}(\xi)\)-adic lengths are determined
exactly by the Hodge numbers of the generic fiber. Comparing minors then controls these lengths
under specialization and yields the desired polygon comparisons. 

\subsection*{Conventions}

\begin{notation}
We adopt Ekedahl's convention in \cite[IV.3.1]{Ekedahl}; see also \cite[6.3]{IllusieFDK}.
For a proper smooth variety $X$ over a perfect field $k$, put $W=W(k)$. For a fixed
$n\geqslant0$, $m_{\lambda}$ is the multiplicity of the Newton slope $\lambda$ of
$\mathrm{H}^n_{\cris}(X/W)[1/p]$. The numbers
\(m^{i,n-i}\coloneqq \sum_{\lambda\in[i-1,i)}(\lambda-i+1)m_{\lambda}
+\sum_{\lambda\in[i,i+1)}(i+1-\lambda)m_{\lambda}\) are the slope numbers of $X$.
The domino number $T^{i,n-i}$ is the dimension of the domino attached to
$d:\mathrm{H}^{n-i}(X,W\Omega^i)\to \mathrm{H}^{n-i}(X,W\Omega^{i+1})$, and
\(h_W^{i,n-i}\coloneqq m^{i,n-i}+T^{i,n-i}-2T^{i-1,n-i+1}+T^{i-2,n-i+2}\)
are the Hodge--Witt numbers. All these sequences have finite support.
The $T^{i,n-i}$ are nonnegative, whereas the $h_W^{i,n-i}$ need not be.
\end{notation}

\begin{notation}
Let $\{u^i\}_{i\in\mathbb Z}$ be a finitely supported sequence of real numbers, with
$b\coloneqq\sum_i u^i\geqslant0$. For $a\in\mathbb Z$, set
\begin{equation}\label{eq:weightedcuts}
 S_a\{u^i\}=\sum_{l\leqslant a}(a-l)u^l.
\end{equation}
We define the polygon associated with $\{u^i\}$ to be the lower boundary of the closed
convex region
\[
 E\{u^i\}\coloneqq
 \bigcap_{a\in\mathbb Z}
 \bigl\{(x,y)\in\mathbb R^2:y\geqslant ax-S_a\{u^i\}\bigr\}.
\]
The defining inequalities automatically imply $0\leqslant x\leqslant b$.

When $u^i\geqslant0$ for all $i$, this is the usual polygon with slope $i$ of horizontal
length $u^i$. The definition also applies when some entries are negative.
Applied to $\{h^{i,n-i}\}$ or $\{h_W^{i,n-i}\}$, it defines the Hodge polygon and
Hodge--Witt polygon used here.
\end{notation}

\begin{lemma}\label{lem:support}
Let $\{u^i\}_{i\in\mathbb Z}$ and $\{v^i\}_{i\in\mathbb Z}$ be two finitely supported
sequences of real numbers with $\sum_i u^i=\sum_i v^i=b\geqslant0$.
If $S_a\{u^i\}\leqslant S_a\{v^i\}$ for every $a\in\mathbb Z$, then the polygon of
$\{u^i\}$ lies on or above the polygon of $\{v^i\}$. If moreover $u^i\geqslant0$ for all
$i$, the converse holds.
\end{lemma}
\begin{proof}
The inequalities imply $E\{u^i\}\subseteq E\{v^i\}$ by comparing the defining halfspaces. Conversely, since
$u^i\geqslant0$, the supporting line of each integral slope $a$ meets the polygon of $\{u^i\}$, possibly
at an endpoint. Thus
\(\inf_{(x,y)\in E\{u^i\}}(y-ax)=-S_a\{u^i\}.\)
If $E\{u^i\}\subseteq E\{v^i\}$, this infimum is at least $-S_a\{v^i\}$, proving the converse. The entries of
$\{v^i\}$ may have either sign.
\end{proof}

\begin{lemma}\label{lem:telescope}\label{eq:hinge}
For every $a\in\mathbb Z$,
\begin{equation}\label{eq:hingem}
 S_a\{m^{i,n-i}\}=\sum_{\lambda\leqslant a}(a-\lambda)m_{\lambda},\qquad S_a\{h_W^{i,n-i}\}=\sum_{\lambda\leqslant a}(a-\lambda)m_{\lambda}+T^{a-1,n-a+1}.
\end{equation}
\end{lemma}
\begin{proof}
A Newton slope $\lambda$ contributes $\max(a-\lambda,0)$ to the first sum. In the second sum,
the coefficient of $T^{j,n-j}$ is \(\max(a-j,0)-2\max(a-j-1,0)+\max(a-j-2,0)\), which equals \(1\) if \(j=a-1\) and \(0\) otherwise. 
\end{proof}

Let \(b_n\) be the \(n\)-th Betti number; put \(t_{N,n}=\sum_\lambda\lambda m_\lambda.\)
A calculation gives
\(\sum_i h_W^{i,n-i}=b_n\) and \(\sum_i i h_W^{i,n-i}=t_{N,n}.\)
The domino terms cancel in both sums. Hence the Hodge--Witt polygon 
has the same endpoint as the Newton polygon. In
particular, it is a single point when $b_n=0$.

\begin{notation}
  In this article, we let \[\ell(M)=\ell_W(M)\] be the length of a \(W\)-module. 
  For a \(W\)-complex \(P\), we write \[\ell^n(P)=\ell^n_W(P)\coloneqq \ell_W(\HH^n(P))\] for the
length of the \(n\)-th cohomology of \(P\). In general, for a discrete 
  valuation ring \(V\) and a \(V\)-module \(M\) or a \(V\)-complex \(P\), we let 
  \(\ell_V(M)\) be the length of \(M\) and \(\ell_V^n(P)\) be the length of the \(n\)-th cohomology of \(P\).
\end{notation}

\section{Specialization in equal characteristic $p$}\label{sec:equalchar}

We prove Theorem~\ref{thm:equalchar} using the de~Rham--Witt complex over a perfect base.
The proof compares the lengths of finite Nygaard quotients and then identifies their leading terms
by Ekedahl's theory of coherent Raynaud complexes.

Let $S=\operatorname{Spec}(B)$, where $B$ is a perfect $k$-algebra, and let $f:X\to S$ be proper
and smooth. Let $W(B)$ be its ring of Witt vectors, and write $\sigma$ for its Witt Frobenius. 
Illusie's de~Rham--Witt theory applies over the perfect base $S$:
$W\Omega^\bullet_{X/S}$ is a natural $W(B)$-linear complex of sheaves on the small \'etale site of $X$ \cite[I, 1.12--1.15 and \S3]{IllusieDRWCrystalline}. Its operators
$p,F,V$ are injective \cite[Proposition~2.1]{IllusieDRW}. Let $\sigma_*$ denote restriction of scalars along $\sigma$; modulo $p$, we use the same
notation for Frobenius on $B$.

\begin{notation}
For $r\geqslant1$ and $a\geqslant0$, let us define a subcomplex of \(W\Omega^\bullet_{X/S}\), 
which we will see corresponds exactly to the $r$-Nygaard filtration 
of Andriopoulos \cite[\S5]{Andriopoulos}:
\begin{equation}\label{eq:rthnygaard}
 (\Nyg^{\geqslant a}_rW\Omega^\bullet_{X/S})^i=
 \begin{cases}
 p^{r(a-i-1)}V^rW\Omega^i_{X/S}&i<a,\\
 W\Omega^i_{X/S}&i\geqslant a.
 \end{cases}
\end{equation}
The relation $V^rd=p^rdV^r$ shows that \(\Nyg^{\geqslant a}_rW\Omega^\bullet_{X/S}\) is indeed 
a subcomplex of \(W\Omega^\bullet_{X/S}\).
 Let us use \(\iota: \Nyg^{\geqslant a}_rW\Omega^\bullet_{X/S}\to W\Omega^\bullet_{X/S}\) to denote the inclusion. 
For $r=1$, this is exactly the usual Nygaard filtration; see \cite[Definition 8.1]{BMS2}. 
Although \cite[\S8.1]{BMS2} treats smooth algebras over a perfect field, the definition and
the proof of Lemma~8.2 apply over any perfect base. 
From now on we abbreviate \(\Nyg^{\geqslant a}_rW\Omega^\bullet_{X/S}\) as \(\Nyg^{\geqslant a}_r\)
when there is no ambiguity. We also write \(\Nyg\) for \(W\Omega^\bullet_{X/S}\) when convenient. 

By the standard formula \(dF=pFd\), the map \(p^iF: W\Omega^i_{X/S}\to W\Omega^i_{X/S}\) 
defines a map of complexes which we call the crystalline Frobenius
$\Phi:W\Omega^\bullet_{X/S}\to\sigma_*W\Omega^\bullet_{X/S}$. 
The divided Frobenius is thus defined as the map of complexes 
\[\Phi/p^a:\Nyg^{\geqslant a}_1\to
\sigma_*W\Omega^\bullet_{X/S}.\]
\end{notation}

The graded pieces for $r=1$ admit the functorial quasi-isomorphisms
\begin{equation}\label{eq:wittnygaardgraded}
 \gr^a_1 W\Omega^\bullet_{X/S}
 :=\Nyg^{\geqslant a}_1/\Nyg^{\geqslant a+1}_1
 \simeq\tau_{\leqslant a}\bigl(\sigma_*\Omega^\bullet_{X/S  }\bigr),
 \qquad a \geqslant0,
\end{equation}
where $\tau_{\leqslant a}$ is canonical truncation of the complex of sheaves; see \cite[Lemma 8.2]{BMS2}.

\begin{proposition}\label{rth-nygaard}
  The $r$-Nygaard filtration \(\Nyg^{\geqslant a}_r\) can be interpreted as 
  \begin{equation}\label{eq:Nygar}
 \Nyg^{\geqslant a}_r=\Nyg^{\geqslant a}_1 
       \times^{h}_{\sigma_*\Nyg}\left(\sigma_*\Nyg^{\geqslant a}_1\right)
       \times^{h}_{\sigma_*^2\Nyg}\cdots
       \times^{h}_{\sigma_*^{r-1}\Nyg}\left(\sigma_*^{r-1}\Nyg^{\geqslant a}_1\right).
\end{equation}
Here, at the $j$-th link, for $1\leqslant j\leqslant r-1$, the two maps to $\sigma_*^j\Nyg$ are the divided Frobenius
$\sigma_*^{j-1}(\Phi/p^a)$ from the left and the inclusion $\sigma_*^{j}(\iota)$ from the right.
\end{proposition}
See also \cite[Theorem~1.1 and Definition~5.1]{Andriopoulos}.
\begin{proof}
Frobenius twisting preserves finite homotopy limits, and each pair of matching maps in
(\ref{eq:Nygar}) is a twist of the classical divided Frobenius and the Nygaard inclusion. For $i<a$, the divided map \(\Phi/p^a: p^{a-i-1}VW\Omega^i_{X/S}\rightarrow W\Omega^i_{X/S}\) is surjective, since
\(\Phi/p^a(p^{a-i-1}Vy)=p^{i-a}F(p^{a-i-1}Vy)=y\); for $i\geqslant a$, the matching inclusion is the identity. The difference map defining the homotopy
pullback is therefore termwise surjective, so the strict pullback represents it. 
Suppose in degree \(i<a\) a local section on the right hand side of (\ref{eq:Nygar}) is represented by \((x_0,x_1,...,x_{r-1})\). 
Then the homotopy pullback condition gives \(x_{\ell}=p^{a-i-1}Vx_{\ell+1}\) for all $0\leqslant \ell\leqslant r-2$. 
Iterating this relation gives \(x_0=p^{(r-1)(a-i-1)}V^{r-1}x_{r-1}\). Since \(x_{r-1}\) is itself locally of the form \(p^{a-i-1}Vy\), it follows that \(x_0\) is locally
of the form \(p^{r(a-i-1)}V^{r}y\). For $i\geqslant a$, the first coordinate is
arbitrary and determines the remaining coordinates. Hence projection to the first factor identifies
the iterated homotopy pullback with the subcomplex in \eqref{eq:rthnygaard}, giving (\ref{eq:Nygar}). 
\end{proof}

Proposition \ref{rth-nygaard} shows that on the canonical representative \(W\Omega^{\bullet}_{X/S}\) of \(\varphi^*\Prism_{X/S}\), 
Andriopoulos's \(r\)-th Nygaard filtration \(\Nyg^{\geqslant a}_r\varphi^*\Prism_{X/S}\) is exactly represented by
\(\Nyg^{\geqslant a}_rW\Omega^{\bullet}_{X/S}\). But for this section, 
working directly with the de Rham--Witt complex and its Nygaard filtration is sufficient for our purposes.

For $a,r\geqslant1$, we set
\begin{equation}\label{calQcalP}
   \calQ^a_{r,X/S}=W\Omega^\bullet_{X/S}/\Nyg^{\geqslant a}_rW\Omega^\bullet_{X/S},\qquad
  \calP^a_{r,X/S}=\RG(X,\calQ^a_{r,X/S}).
\end{equation}

\begin{lemma}\label{lem:wittfiltration}
For $r,\ell\geqslant1$, there is a short exact sequence of complexes
\[
 0\longrightarrow\sigma_*^r\calQ^a_{\ell,X/S}
 \longrightarrow\calQ^a_{r+\ell,X/S}
 \longrightarrow\calQ^a_{r,X/S}\longrightarrow0.
\]
Thus $\calQ^a_{r,X/S}$ has a finite filtration with factors
$\sigma_*^\ell\calQ^a_{1,X/S}$, $0\leqslant \ell<r$.
\end{lemma}
\begin{proof}
The divided Frobenius $\Phi^r/p^{ra}$ composed with the quotient map induces a map
\[
 \Nyg^{\geqslant a}_rW\Omega^\bullet_{X/S}
 \rightarrow \sigma_*^rW\Omega^\bullet_{X/S}\rightarrow \sigma_*^r\calQ^a_{\ell,X/S}.
\]
In degree $i<a$, it sends $p^{r(a-i-1)}V^ry$ to the class of $y$; in degrees $i\geqslant a$, the
target vanishes. It is therefore surjective, with kernel
$\Nyg^{\geqslant a}_{r+\ell}W\Omega^\bullet_{X/S}$. Taking the corresponding quotients gives
the asserted exact sequence. 
\end{proof}

\begin{lemma}[Perfectness and base change]\label{lem:wittperfectness}
For every $a,r\geqslant1$, the complex $\calP^a_{r,X/S}$ is a perfect $W(B)$ complex and satisfies
$\calP^a_{r,X/S}[1/p]=0$. For every homomorphism $B\to B'$ of perfect $k$-algebras, put
$S'=\operatorname{Spec}(B')$ and $X'=X\times_S S'$. The natural base-change map is an equivalence
\begin{equation}\label{eq:equalcharbasechange}
 \calP^a_{r,X/S}\otL_{W(B)}W(B')\ \xrightarrow{\ \sim\ }\ \calP^a_{r,X'/S'}.
\end{equation}
\end{lemma}
\begin{proof}
For $r=1$, the finite Nygaard filtration on $\calQ^a_{1,X/S}$ has graded pieces
\eqref{eq:wittnygaardgraded}, with $0\leqslant b<a$. By the Cartier isomorphism, their
Postnikov filtrations have successive quotients $\Omega^i_{X/S}[-i]$, with $0\leqslant i\leqslant b$. Since $X/S$ is proper and
smooth, each $\RG(X,\Omega^i_{X/S})$ is perfect over $B$ and satisfies arbitrary coherent derived
base change \cite[Tag~0DJT]{Stacks}. Since $B=W(B)/p$ and $p$ is a nonzerodivisor in $W(B)$, these
complexes are also perfect over $W(B)$. The two finite filtrations therefore show that $\calP^a_{1,X/S}$ is
perfect over $W(B)$.

For base change, since $B'$ is perfect, $p$ is also a nonzerodivisor in $W(B')$, so
\(B\otL_{W(B)} W(B')\simeq[W(B')\xrightarrow{p}W(B')]\simeq B'.\)
Consequently, coherent derived base change gives
\[
 \RG(X,\Omega^i_{X/S})\otL_{W(B)} W(B')
 \simeq\RG(X,\Omega^i_{X/S})\otL_B B' 
 \simeq\RG(X',\Omega^i_{X'/S'}).
\]
The graded comparison and Cartier isomorphisms are functorial, so these equivalences identify the
base-change map on every piece of the two filtrations. This proves
\eqref{eq:equalcharbasechange} for $r=1$. Lemma~\ref{lem:wittfiltration} gives perfectness and
base change for every $r$, since Frobenius restriction commutes with base change along maps of
perfect rings. Finally, $p^{ra}$ annihilates $\calQ^a_{r,X/S}$ by $p^r=V^rF^r$, so
$\calP^a_{r,X/S}[1/p]=0$.
\end{proof}

\subsection{Length growth over a perfect field}\label{subsec:asymptotic}

We now identify the leading terms of the fiber lengths. The same calculation will also be used
in the mixed-characteristic comparison. 

\begin{notation}
Fix a perfect field $k$, put $W=W(k)$, and let $R$ be the
Raynaud ring with Witt Frobenius $\sigma$ \cite[I, \S1.1]{IR}. Its relevant relations include
\[
 FV=VF=p,\qquad FdV=d,\qquad Vd=pdV,\qquad d^2=0.
\]
We adopt the notation used in \cite{Ekedahl} for left graded $R$-modules and \(R\)-complexes. 
In particular, let \(D(R)\) denote the derived category of left graded \(R\)-modules and
\(D_c^b(R)\) its full subcategory of bounded coherent Raynaud complexes; see \cite[I, Remark~3.10.1]{IR} and
\cite[Definition~0.5.1]{Ekedahl}.
For a proper smooth variety \(X/k\), its de Rham--Witt complex 
\(W\Omega^\bullet_{X/k}\) is a sheaf of left graded \(R\)-modules on the small \'etale site of \(X\).  
Its derived global sections form the Hodge--Witt cohomology complex
\(\RG(X,W\Omega^\bullet_{X/k})\in D(R)\). 
It is a theorem of Illusie and Raynaud that 
\(\RG(X,W\Omega^\bullet_{X/k})\) always lies in \(D_c^b(R)\) \cite[II, Theorem~2.2]{IR}. 

We recall that there is a well-defined totalization functor \(\Tot:D_c^b(R)\to D(W)\) 
which is exact and takes the Hodge--Witt cohomology of \(X/k\) to its crystalline cohomology.

Let $M\in D_c^b(R)$ be a bounded coherent Raynaud complex. Internal degree is denoted by $i$,
external cochain degree by $b$, and total degree by $i+b$. 
We may represent $M$ by a complex bounded in both degrees \cite[Proposition~0.5.2]{Ekedahl}. 
\end{notation}

\begin{notation}
Let \(M\in D(R)\). 
Let $\delta_a:\mathbb{Z}\to\mathbb{Z}$ be the characteristic function of $\{a\}$. The Nygaard modification \cite[Chapter~0, \S2]{Ekedahl}
\[
 M(-r\delta_a)
\]
is defined by
\begin{equation}\label{eq:abstractN}
 (M(-r\delta_a))^i=
 \begin{cases}\sigma_*^r M^i&i<a,\\M^i&i\geqslant a,\end{cases}
\end{equation}
where the operators $F,V$ are induced from $M$, and the internal differential is $d$ except from degree $a-1$
to degree $a$, where it is $dV^r$; the external differential is unchanged. Here $c\in W$ acts on
$\sigma_*^rM$ as $\sigma^r(c)$, so $V^r:\sigma_*^rM\to M$ is $W$-linear. The Raynaud relations, in
particular $FdV=d$, make this a complex of graded $R$-modules. Nygaard modification is exact and
preserves coherence, hence defines a functor on $D_c^b(R)$.

There is a natural morphism $\iota_r^a:M(-r\delta_a)\to M$ whose component in internal
$R$-module degree $i$ is
\begin{equation}\label{eq:iota}
 (\iota_r^a)^i=
 \begin{cases}p^{r(a-i-1)}V^r&i<a,\\1&i\geqslant a.\end{cases}
\end{equation}
The relation $V^rd=p^rdV^r$ gives compatibility with the internal differential, and the components
commute with $F,V$ and the external differential. Thus $\iota_r^a$ is a morphism of Raynaud
complexes.
\end{notation}

For a smooth proper scheme $X/k$, the sheafwise modification $W\Omega^\bullet_X(-r\delta_a)$ is
identified by $\iota_r^a$ with the subcomplex $\Nyg^{\geqslant a}_rW\Omega^\bullet_X$ in
\eqref{eq:rthnygaard}, since $p$ and $V$ are injective. Thus the two Nygaard constructions
are compatible. Nygaard modification also commutes with derived global sections of sheaves of graded
$R$-modules, yielding a natural isomorphism in $D_c^b(R)$:
\begin{equation}\label{eq:nygaardglobal}
 \RG(X,W\Omega^\bullet_{X})(-r\delta_a)
 \simeq\RG(X,\Nyg^{\geqslant a}_rW\Omega^\bullet_{X}).
\end{equation}

For \(M\in D_c^b(R)\), define
\[
 Q^a_r(M)=\cofib\bigl(M(-r\delta_a)
       \xrightarrow{\iota_r^a}M\bigr)\in D_c^b(R),\qquad P^a_r(M)=\Tot(Q^a_r(M))\in D(W).
\] 
Since Nygaard modification and finite internal totalization are exact, $P^a_r$ is a triangulated
functor $P^a_r:D_c^b(R)\rightarrow D(W)$. 
When \(M\) is the Hodge--Witt complex of a smooth proper scheme \(X/k\), 
the cofiber \(P^a_r(\RG(X,W\Omega^\bullet_{X/k}))\) is identified with \(\calP^a_{r,X/k}\) of \eqref{calQcalP}.

For the de~Rham--Witt complex, this is the modified complex of \cite[Definitions~1.1--1.2,
pp.~370--371]{Nygaard}, up to a Frobenius twist.

For $M\in D_c^b(R)$, we use Ekedahl's slope, domino, and Hodge--Witt numbers
\cite[IV.3.1]{Ekedahl}, with Frobenius $p^iF$ in internal degree $i$ and the same
second-difference convention as above.

\begin{proposition}[Length growth]\label{prop:growth}
Let $M\in D_c^b(R)$ and fix an integer $a$. Let
\[\ell^n(P^a_r(M))\coloneqq \ell_W(\mathrm{H}^n(P^a_r(M))).\] For every $n$, there is a
constant $C_{M,a,n}\geqslant0$, independent of $r\geqslant1$, such that
\[
 \left|\ell^n(P^a_r(M))
       -r\,S_a\{h_W^{i,n-i}(M)\}\right|
       \leqslant C_{M,a,n}.
\]
\end{proposition}

We prove the proposition using two finite filtrations on objects of $D_c^b(R)$. The first comes from
Ekedahl's diagonal $t$-structure, with heart $\dc$ \cite[Chapter~I, \S1]{Ekedahl}; we write
$\taut_{\leqslant \ell}$ and $\Ht^\ell$ for its truncation and cohomology functors. The second is
the usual Postnikov filtration of an $R$-complex in the external cohomological degree, with
successive quotients $\mathrm{H}^b(M)[-b]$. The two truncations commute
\cite[Proposition~0.1.4]{Ekedahl}.

For any $M\in D_c^b(R)$, each ordinary cohomology module $\mathrm{H}^b(M)$ admits a finite
filtration whose factors are internal grading shifts of Dieudonn\'e modules finite over $W$
(the type-I factors) and elementary dominos \cite[Chapter~0, \S4]{Ekedahl}. All these filtrations are independent of $r$. We
first compute $Q^a_r$ on the elementary factors.

\begin{lemma}\label{lem:typeIgrowth}
Let $E\in D_c^b(R)$ be a finite-type Dieudonn\'e module concentrated in internal degree $j$ and
external degree $n-j$. Then the conclusions of Proposition~\ref{prop:growth} hold for $E$. Moreover,
$\mathrm{H}^q(P^a_rE)=0$ for $q\notin\{n-1,n\}$.
\end{lemma}

\begin{proof}
Put $b=n-j$. For $j\geqslant a$, the cone vanishes. For $j<a$,
\(P^a_r(E)\simeq\bigl[\sigma_*^rE\xrightarrow{\iota_r^a}E\bigr]\) where \(\iota_r^a=p^{r(a-j-1)}V^r\) 
in total degrees $n-1,n$. Its cohomology is $\ker(\iota_r^a)$ and $\coker(\iota_r^a)$ in these
degrees, respectively.

Suppose first that $E$ is torsion-free. Its only possibly nonzero Hodge--Witt numbers admit
the description \(h_W^{j,b}(E)=\ell_W(E/VE)\) and \(h_W^{j+1,b-1}(E)=\ell_W(E/FE)\);
see \cite[Lemma~2(ii) and the proof of Lemma~3]{Crew}. Since $FV=p$, these lengths sum to $\rk_W E$.
The map $\iota_r^a$ is injective, and additivity of colength gives
\begin{align*}
 \ell^n(P^a_rE)
 &=r\bigl((a-j-1)\rk_W E+\ell_W(E/VE)\bigr)\\
 &=r\sum_{i<a}(a-i)h_W^{i,n-i}(E).
\end{align*}

For general $E$, put $e(E)=\ell_W(E_{\mathrm{tor}})$ and $\overline E=E/E_{\mathrm{tor}}$. The snake
lemma, applied to the torsion sequence and $\iota_r^a$, bounds both the kernel on $E$ and the
difference of cokernel lengths on $E$ and $\overline E$ by $e(E)$. These modules have the same
Hodge--Witt numbers, since their rational slopes agree and their domino terms vanish. Thus
\begin{gather*}
 0\leqslant \ell^n(P^a_rE)-rS_a\{h_W^{i,n-i}(E)\}\leqslant e(E),\\
 \ell^{n-1}(P^a_rE)\leqslant e(E).
\end{gather*}
All other cohomology vanishes. For $j\geqslant a$, the same bounds hold with
$S_a\{h_W^{i,n-i}(E)\}=0$. Perfection follows since $W$ is a DVR.
\end{proof}

\begin{lemma}\label{lem:dominogrowth}
Let $U=U_t\in D_c^b(R)$ be an elementary domino of type $t$, concentrated in external degree $n-j$
and internal degrees $j,j+1$. Then the conclusions of Proposition~\ref{prop:growth} hold for $U$.
Moreover, $\mathrm{H}^q(P^a_rU)=0$ for $q\notin\{n-1,n,n+1\}$.
\end{lemma}

\begin{proof}
The source of $U$ is $k[[V]]$, its target is killed by $V$, and $pU=0$. Its total cohomology has lengths
\(\ell^n(\Tot U)=\max(t,0)\) and \(\ell^{n+1}(\Tot U)=\max(-t,0),\)
and vanishes otherwise.

For $a\leqslant j$, the cone vanishes. For $a\geqslant j+2$, the modification is a scalar twist and
$\iota_r^a=0$, since $p$ kills the source and $V$ kills the target. Hence
\(P^a_r(U)\simeq\Tot U\oplus\sigma_*^r\Tot U[1],\)
whose cohomology is supported in $[n-1,n+1]$ and has length at most $|t|$ in each degree. At
$a=j+1$, the components of $\iota_r^a$ are $V^r$ and the identity. This map is injective, with
quotient $k[[V]]/(V^r)$ in total degree $n$, so
\( P^a_r(U)\simeq\bigl(k[[V]]/(V^r)\bigr)[-n]\), and \(\ell^n(P^a_rU)=r.\)
The only nonzero Hodge--Witt numbers are
\(h_W^{j,n-j}(U)=1, h_W^{j+1,n-j-1}(U)=-2,h_W^{j+2,n-j-2}(U)=1.\)
Thus \[
 S_a\{h_W^{i,n-i}(U)\}
 =\sum_{i<a}(a-i)h_W^{i,n-i}(U)
 =\begin{cases}
 1,&a=j+1,\\
 0,&a\neq j+1.
 \end{cases}
\]
It follows that Proposition~\ref{prop:growth} holds, with error and off-degree lengths bounded by $|t|$; both
vanish at $a=j+1$. The resulting bounded complexes with finite-length cohomology are perfect over
$W$.
\end{proof}

\begin{lemma}\label{lem:heartgrowth}
Let $D\in\dc$ and $n\in\mathbb{Z}$. Then $M=D[-n]$ satisfies the conclusions of
Proposition~\ref{prop:growth}. More precisely, for each fixed $a$,
\begin{equation}\label{eq:heartgrowth}
 \begin{aligned}
 \ell^n(P^a_rM)&=rS_a\{h_W^{i,n-i}(M)\}+O(1),\\
 \ell^{n+1}(P^a_rM)&=O(1), \quad\ell^{n-1}(P^a_rM)=O(1)\\
 \ell^q(P^a_rM)&=0\qquad\bigl(q\notin\{n-1,n,n+1\}\bigr),
 \end{aligned}
\end{equation}
where the bounds are independent of $r$.
\end{lemma}

\begin{proof}
Since $P^a_r$ commutes with external shifts, it suffices to treat $n=0$. By
\cite[Proposition~0.1.4]{Ekedahl} and \cite[I, Theorem~3.8 and Proposition~2.15]{IR}, there is a
finite filtration in $\dc$
\[
 0=D_0\subset D_1\subset\cdots\subset D_m=D,
 \qquad E_\nu=D_\nu/D_{\nu-1},
\]
independent of $r$, whose factors $E_\nu$ are concentrated in an external degree $\ell_\nu$ and are
either finite-type Dieudonn\'e modules in internal degree $-\ell_\nu$ or elementary dominos in
internal degrees $-\ell_\nu,1-\ell_\nu$. Lemmas~\ref{lem:typeIgrowth} and~\ref{lem:dominogrowth}
apply with source total degree zero.

For a triangle $A'\to B'\to C'\to A'[1]$ with finite-length cohomology, the connecting maps
$\delta_q:\mathrm{H}^q(C')\to \mathrm{H}^{q+1}(A')$ give
\begin{equation}\label{eq:leslength}
 \ell^q(B')=\ell^q(A')+\ell^q(C')
 -\ell_W(\im(\delta_{q-1}))- \ell_W(\im(\delta_q)).
\end{equation}
In degree zero, the two correction terms have bounded length, since one has bounded source
and the other bounded target. The elementary estimates and these long exact sequences give,
by finite induction,
\[
 \begin{aligned}
 \ell^q(P^a_rD)&=0 &&\bigl(q\notin\{-1,0,1\}\bigr),\\
 \ell^1(P^a_rD)&=O(1),\quad \ell^{-1}(P^a_rD)=O(1),\\
 \ell^0(P^a_rD)&=r\sum_\nu S_a\{h_W^{i,-i}(E_\nu)\}+O(1).
 \end{aligned}
\]
All errors are bounded independently of $r$. Additivity of Hodge--Witt numbers in $\dc$
\cite[Proposition~I.1.4(i), III, (2.6)(i)]{Ekedahl} identifies the sum with
$S_a\{h_W^{i,-i}(D)\}$; the Hodge--Witt numbers of $D$ vanish outside total degree zero.
This proves \eqref{eq:heartgrowth}. The same filtration gives finite-length cohomology in a bounded
range, hence perfectness over $W$.
\end{proof}

\begin{proof}[Proof of Proposition~\ref{prop:growth}]
Consider the diagonal cohomology $\Ht^j(M)$ of $M$ and the truncations $\taut_{\leqslant j}M$.
Recall that
\(h_W^{i,n-i}(M)=h_W^{i,n-i}(\Ht^n(M)[-n])\) for every \(i\).
By Lemma~\ref{lem:heartgrowth}, $P^a_r(\Ht^j(M)[-j])$ has cohomology in degrees $[j-1,j+1]$, with
bounded lengths outside degree $j$, and
\[
 \ell^j(P^a_r\Ht^j(M)[-j])
 =rS_a\{h_W^{i,j-i}(M)\}+O(1).
\]
Apply $P^a_r$ to the finite diagonal Postnikov tower
\[
 \taut_{\leqslant j-1}(M)\rightarrow \taut_{\leqslant j}(M)\rightarrow \Ht^j(M)[-j]\xrightarrow{+1} .
\]
Every connecting map $\mathrm{H}^q(P^a_r \Ht^j(M)[-j])\to \mathrm{H}^{q+1}(P^a_r \taut_{\leqslant j-1}(M))$ has image of
uniformly bounded length: for $q\ne j$ its source has uniformly bounded length, while for $q=j$ its target
vanishes, since all earlier layers have cohomology in degrees at most $j$. Thus, in degree $n$,
\eqref{eq:leslength} and finite induction give
\[
 \ell^n(P^a_rM)
 =\sum_j\ell^n(P^a_r\Ht^j(M)[-j])+O(1)
 =rS_a\{h_W^{i,n-i}(M)\}+O(1).
\] 
\end{proof}

\begin{corollary}\label{cor:geometricgrowth}
Let $X$ be smooth and proper over a perfect field $k$, and put $W=W(k)$. For fixed integers $n$ and
$a\geqslant1$,
\begin{equation}\label{eq:geometricgrowth}
 \ell^n\bigl(\RG\bigl(X,W\Omega^\bullet_X/\Nyg^{\geqslant a}_rW\Omega^\bullet_X\bigr)\bigr)
 =rS_a\{h_W^{i,n-i}(X)\}+O(1),
\end{equation}
with error bounded independently of $r\geqslant1$.
\end{corollary}
\begin{proof}
The Raynaud complex $\RG(X,W\Omega^\bullet_X)$ lies in $D_c^b(R)$ by Illusie--Raynaud's coherence
theorem \cite[II, Theorem~2.2]{IR}. Apply Proposition~\ref{prop:growth} and
\eqref{eq:nygaardglobal}. The sheafwise injectivity of $p$ and $V$ identifies the derived cone with
the strict quotient in \eqref{calQcalP}.
\end{proof}

\subsection{Hodge--Witt polygons under specialization}\label{subsec:hwspecialization}
We next prove a semicontinuity statement for perfect complexes over Witt rings. The following
elementary calculation expresses each fiber length in terms of minors.

\begin{lemma}[Cohomology lengths and minors]\label{lem:minors}
Let $V$ be a discrete valuation ring with uniformizer $\pi$, and let $C^\bullet$ be a bounded complex of finite free $V$-modules that becomes
acyclic over $\Frac(V)$. If $t$ is the rank of $\operatorname{im}(d^{n-1})$, then
\begin{equation}\label{eq:minorlength}
 \ell_V^n(C^\bullet)=v_{\pi}\bigl(\mathrm{I}_t(d: C^{n-1}\to C^n)\bigr),
\end{equation}
where $\mathrm{I}_t$ is the ideal generated by $t$-by-$t$ minors and $\mathrm{I}_0=V$.
\end{lemma}
\begin{proof}
The kernel $Z^n=\ker(d^n)$ is saturated in $C^n$, since $C^n/Z^n=\im(d^n)$ is torsion-free, and
hence is a direct summand. The nonzero Smith factors of $d^{n-1}$ in $C^n$ are therefore the same as
those in $Z^n$. The sum of their valuations is the length of $Z^n/\im(d^{n-1})$. If $t=0$, rational
acyclicity forces $Z^n=0$.
\end{proof}

\begin{lemma}[Semicontinuity of Witt-vector lengths]\label{lem:equalcharlength}
Let $B$ be a perfect $\mathbb F_p$-algebra, and let $\calP$ be a perfect $W(B)$-complex
such that $\calP[1/p]$ is acyclic. For every integer $n$, the function
\[
 \ell^n(\calP,x):=\ell_{W(\kappa(x))}^n
 \bigl(\calP\otL_{W(B)} W(\kappa(x))\bigr),
 \qquad x\in\operatorname{Spec}(B),
\]
is upper semicontinuous. More precisely, for every integer $m\geqslant0$, the locus
\[\{x:\ell^n(\calP,x)\geqslant m\}\] is Zariski closed and is locally the zero set of a finitely
generated ideal. In particular, if $y$ is a specialization of $x$, then
\begin{equation}\label{eq:equalcharfiniteineq}
 \ell^n(\calP,y)\geqslant\ell^n(\calP,x).
\end{equation}
The lengths are unchanged by extending the residue fields to perfect fields.
\end{lemma}

\begin{proof}
Represent $\calP$ by a bounded complex $C^\bullet$ of finite projective $W(B)$-modules. Zariski locally
on $\operatorname{Spec}(B)$, we may assume its terms are free: choose a principal open $D(f)$
on which the terms of $C^\bullet/p$ are free, and base change to $W(B_f)$. A basis modulo $p$ thus
lifts to a basis over $W(B_f)$.

For each point $x\in\operatorname{Spec}(B)$, the residue field $\kappa(x)$ is perfect, so $W(\kappa(x))$ is a discrete valuation ring with uniformizer $p$. Each fiber of
$C^\bullet$ becomes acyclic after inverting $p$. The ranks of the resulting differentials are
therefore determined by the ranks of its terms, starting from the leftmost term. In particular,
$d^{n-1}$ has a fixed rational rank $t$ on this open subset. If $t=0$, then $\ell^n=0$.
Otherwise, let $\beta_1,\ldots,\beta_N\in W(B_f)$ be its $t$-by-$t$ minors.
Lemma~\ref{lem:minors} gives
\[
 \ell^n(\calP,x)=\min_i v_p\bigl(\beta_i(x)\bigr).
\]
For $m\geqslant1$, the condition $\ell^n(\calP,x)\geqslant m$ is precisely the vanishing of the first
$m$ Witt coordinates of every $\beta_i$ at $x$. These finitely many elements of $B_f$ define
the required closed locus. The assertion for $m=0$ is immediate, and closedness gives
\eqref{eq:equalcharfiniteineq}. Finally, an extension of perfect fields induces a flat map of
Witt discrete valuation rings with the same uniformizer $p$, and hence preserves these cohomology lengths.
\end{proof}

We now combine the length semicontinuity of Lemma~\ref{lem:equalcharlength} with
Corollary~\ref{cor:geometricgrowth} to prove Theorem~\ref{thm:equalchar} directly over an
arbitrary base. The finite Nygaard quotients give the following stronger statement for weighted
sums.

\begin{proposition}[Closed loci and specialization]\label{prop:hwclosed}
Let $S$ be a connected $\mathbb F_p$-scheme, 
and let $f:X\to S$ be a proper smooth morphism. Fix $n\geqslant0$. For $x\in S$, write
$X_{\bar x}=X\otimes_S\overline{\kappa(x)}$.
\begin{enumerate}[label=\textup{(\roman*)}]
\item For every $a\in\mathbb Z$ and $c\in\mathbb R$, the locus
\[
 \left\{x\in S:
 S_a\{h_W^{i,n-i}(X_{\bar x})\}\geqslant c\right\} 
\]
is Zariski closed.
\item Suppose the $n$-th Betti number of $X_{\bar x}$ is $b$. For any continuous function
$P:[0,b]\to\mathbb R$, the locus where the degree-$n$ Hodge--Witt polygon of $X_{\bar x}$ lies on
or below $P$ is Zariski closed.
\item The above loci are locally the zero set of a finitely generated ideal. 
\end{enumerate}In particular, for every specialization $x\leadsto y$, 
\(\HW_n(X_{\bar y})\leqslant\HW_n(X_{\bar x})\). 
\end{proposition}

\begin{proof}
For $a\leqslant0$, the weighted sums vanish, so \textup{(i)} is immediate.
Fix $a\geqslant1$ and put $k_x=\overline{\kappa(x)}$. For $r\geqslant1$, set
\[
 \ell^n(\calP_r^a,x)=\ell_{W(k_x)}^n\bigl(\RG(X_{\bar x},\calQ^a_r(X_{\bar x}))\bigr).
\]
We first show that $x\mapsto\ell^n(\calP_r^a,x)$ is upper semicontinuous. Work affine locally on $S$ and replace its
coordinate ring by its perfection $B$; this does not change the underlying topological space.
Write $X$ also for its base change to $\operatorname{Spec}(B)$.
Lemma~\ref{lem:wittperfectness} shows that $\calP^a_{r,X/\operatorname{Spec}(B)}$
is perfect over $W(B)$, vanishes after inverting $p$, and computes the fiber complexes
by derived base change. Lemma~\ref{lem:equalcharlength} therefore gives the upper semicontinuity
of $x\mapsto\ell^n(\calP_r^a,x)$ on $\operatorname{Spec}(B)$, and hence on $S$.

Lemma~\ref{lem:wittfiltration} and the long exact sequence in cohomology give
$\ell^n(\calP_{r+\ell}^a,x)\leqslant \ell^n(\calP_r^a,x)+\ell^n(\calP_\ell^a,x)$; Frobenius twists preserve length. Together with
Corollary~\ref{cor:geometricgrowth}, this yields the exact formula
\begin{equation}\label{eq:cutinfimum}
 S_a\{h_W^{i,n-i}(X_{\bar x})\}
 =\lim_{r\to\infty}\frac{\ell^n(\calP_r^a,x)}r
 =\inf_{r\geqslant1}\frac{\ell^n(\calP_r^a,x)}r.
\end{equation}
Indeed, $\ell^n(\calP_{mr}^a,x)\leqslant m\ell^n(\calP_r^a,x)$, so the limit is at most $\ell^n(\calP_r^a,x)/r$ for every $r$.
Consequently,
\[
 \left\{x:S_a\{h_W^{i,n-i}(X_{\bar x})\}\geqslant c\right\}
 =\bigcap_{r\geqslant1}\{x:\ell^n(\calP_r^a,x)\geqslant rc\}
\]
is closed. 

For \textup{(ii)}, put $c_a(P)=\sup_{u\in[0,b]}(au-P(u))$. By the supporting-line definition,
the degree-$n$ Hodge--Witt polygon lies on or below $P$ precisely when
\[
 S_a\{h_W^{i,n-i}(X_{\bar x})\}\geqslant c_a(P)
 \qquad\text{for every }a\in\mathbb Z.
\]
Its locus is therefore an intersection of the closed subsets in \textup{(i)}. This argument also
applies when some Hodge--Witt numbers are negative.

Finally, affine locally on $S$, the family descends to a proper smooth family over a finitely
generated $\mathbb F_p$-algebra \cite[Tags~01ZM, 0C0C, 081F]{Stacks}. The quantities in the
statement are invariant under extension of algebraically closed fields, by
Lemma~\ref{lem:wittperfectness} and Corollary~\ref{cor:geometricgrowth}. Thus each locus is the
inverse image of the corresponding closed locus on this noetherian model. 
A finitely generated ideal defining the locus on the model pulls back to a finitely generated
ideal on $S$, as required.

For a specialization $x\leadsto y$, the closed loci in \textup{(i)} immediately give
\(\HW_n(X_{\bar y})\leqslant\HW_n(X_{\bar x})\). 
\end{proof}

Combining the weighted-sum comparison with Newton polygon semicontinuity gives the following
corollary.

\begin{corollary}[Semicontinuity of domino numbers]\label{cor:dominospecialization}
Let $S$ be a connected $\mathbb F_p$-scheme. 
Let $f:X\to S$ be a proper smooth morphism. For every $i,j\geqslant0$, the function
$x\mapsto T^{i,j}(X_{\bar x})$ is upper semicontinuous on $S$. Equivalently, for every integer
$m$, the locus
\[
 \{x\in S:T^{i,j}(X_{\bar x})\geqslant m\}
\]
is Zariski closed. Moreover, it is locally the zero locus of a finitely generated ideal. In particular, if $y$ is a
specialization of $x$, then $T^{i,j}(X_{\bar y})\geqslant T^{i,j}(X_{\bar x})$.
\end{corollary}
\begin{proof}
As in the proof of Proposition~\ref{prop:hwclosed}, finite-presentation descent reduces us to an
affine base of finite type over $\mathbb F_p$. Newton polygon semicontinuity applies over this
base, including when it is singular \cite[Theorem~2.1]{CrewSpecialization}. 
In particular, for crystalline Newton polygons in proper smooth families, 
the locus where the degree-$n$ Newton polygon lies on or above 
a fixed polygon is Zariski closed.

Put $n=i+j$ and $a=i+1$. Write $m_\lambda(x)$ for the degree-$n$ Newton slope multiplicities and
set
\[
 \nu_a(x)=\sum_{\lambda<a}(a-\lambda)m_\lambda(x).
\]
Locally on $S$, the Betti number $b_n=b$ is constant. If $P_x$ is the Newton polygon function 
at the point \(x\), then \(\nu_a(x)=\sup_{0\leqslant u\leqslant b}(au-P_x(u)).\) 
Thus $\{x:\nu_a(x)\leqslant c\}$ is the closed locus where $P_x$ lies on or above the line
$u\mapsto au-c$. Lemma~\ref{eq:hinge} gives
\(T^{i,j}(X_{\bar x})=S_a\{h_W^{i,n-i}(X_{\bar x})\}-\nu_a(x).\)
So by Proposition~\ref{prop:hwclosed}, the complement
\[
\{S_a-\nu_a<m\}
=
\bigcup_{q\in\mathbb Q}
\bigl(\{S_a<q\}\cap\{\nu_a>q-m\}\bigr),
\]
is open, so the locus \(\{S_a-\nu_a\geqslant m\}\) is closed. Finally, the same finite-presentation
descent argument as in Proposition~\ref{prop:hwclosed} shows that these loci are defined by finitely
generated ideals.
\end{proof}

For surfaces, this recovers Crew's upper semicontinuity theorem for $T^{0,2}$
\cite[Theorem~7]{Crew}.

\begin{corollary}[Openness of the Hodge--Witt locus]\label{cor:hwopen}
For a family $f:X\to S$ as in Corollary~\ref{cor:dominospecialization}, the subset
\[
 S_{\mathrm{HW}}=\{x\in S:X_{\bar x}\text{ is Hodge--Witt}\}
\]
is Zariski open in $S$.
\end{corollary}
\begin{proof}
Recall that a smooth proper variety $X$ over a perfect field $k$ is Hodge--Witt if all the groups
$\mathrm{H}^j(X,W\Omega_X^i)$ are finitely generated over $W(k)$; equivalently, all its domino
numbers vanish \cite[IV, \S4.6]{IR}. Locally on $S$, the fiber dimensions have a common bound $d$,
so only the indices $0\leqslant i,j\leqslant d$ can contribute. Since domino numbers are
nonnegative integers, Corollary~\ref{cor:dominospecialization} gives
\[
 S_{\mathrm{HW}}
 =\bigcap_{0\leqslant i,j\leqslant d}
   \{x\in S:T^{i,j}(X_{\bar x})\leqslant 0\},
\]
a finite intersection of open subsets.
\end{proof}

Strict jumps of Hodge--Witt polygons seem to be common in characteristic \(p\).
By Ekedahl's K\"unneth formula, products of two Mazur--Ogus varieties and products of two
non-Mazur--Ogus varieties can behave very differently: the former are again Mazur--Ogus,
with the expected Hodge--Witt numbers, whereas the latter often, in Illusie's words,
``burst into many dominos''. A similar contrast can occur when taking the product of a
non-Mazur--Ogus variety \(X\) with an ordinary or a supersingular elliptic curve: the ordinary
factor contributes shifts of the existing cohomology modules, whereas the supersingular factor
can produce additional domino contributions. For example, consider the surface $S$ constructed below, whose Picard scheme has a local-local
part $\alpha_p$. Its Hodge--Witt cohomology \(\HH^2(S,WO_S)\) then has a finite torsion part
\(\mathbb{D}(\alpha_p)\) on which \(F=V=0\). Its product with a supersingular elliptic curve
has more dominos than its product with an ordinary elliptic curve. We compute explicit
examples in Section~\ref{subsec:charpexamples}.

\subsection{Examples of strict specialization}\label{subsec:charpexamples}

We examine three examples. The first is the product construction described above, involving a
surface with non-reduced \(\mathrm{Pic}(S)\) and a family of elliptic curves. It works in every
positive characteristic. The second revisits Ekedahl's computation for the self-product of a
family of Igusa surfaces. The third uses relative Godeaux--Serre approximation to construct
fivefolds with jumping Hodge--Witt polygons.

All examples are over an algebraically closed field \(k\). 

\subsubsection{The family \(S\times_k \mathcal{E}\)}\label{subsubsec:S*E}

The first example is \(S\times_k\mathcal E\), where \(S\) is a surface with
\(\mathrm{Pic}^\circ(S)/\mathrm{Pic}^\circ(S)_{\mathrm{red}}\cong\alpha_p\) and \(\mathcal E\)
is a family of elliptic curves.

For $p>2$, let $C$ be the smooth projective model of \(y^2=x^p-x\), and let
$G=\mathbb Z/p$ act on $C$ with generator \(\beta(x,y)=(x+1,y)\).
For $p=2$, take the supersingular elliptic curve \(y^2+y=x^3+x^2\), with
$G=\mathbb Z/4$ generated by \(\beta(x,y)=(x+1,y+x+1)\).
In both cases, $\beta$ fixes the point $c_0$ at infinity and
$C/\langle\beta\rangle\simeq\mathbb P^1_k$.

Choose an ordinary elliptic curve $A/k$ with a point $a\in A(k)$ of the same order as
$\beta$. The diagonal action of $G$ on $C\times A$, given by $\beta$ on $C$ and translation
by $a$ on $A$, is free. Set \(S\coloneqq(C\times A)/G\). For $p=2,3$, these are the
hyperelliptic examples of \cite[\S\S1.2--1.3]{Lang95}.
Put $B=A/\langle a\rangle$. Projection induces $f:S\to B$, and the fixed point $c_0$
defines a section $s:B\to S$.

\begin{lemma}\label{Pic(S)}
  \(\mathrm{Pic}^\circ(S)\cong B^\vee \times \alpha_p\) and \(\mathrm{NS}(S)\) is free. 
\end{lemma}

\begin{proof}
Put $J=\operatorname{Pic}^0(C)$. We first show that the scheme-theoretic fixed subgroup
$J^G$ is isomorphic to $\alpha_p$. For $p>2$, the Artin--Schreier presentation $x^p-x=y^2$ has a single
pole of order $2$, so $C$ has genus $(p-1)/2$ and $p$-rank zero
\cite[Lemma~2.6]{PriesZhu}. Since $C/G\simeq\mathbb P^1$, the characteristic polynomial of
$\beta^*$ on the rational Tate module of $J$ is $1+T+\cdots+T^{p-1}$. Hence
$\deg(1-\beta^*)=p$. As $J$ and its dual have $p$-rank zero, this order-$p$ kernel can only be $\alpha_p$.
For $p=2$, using the canonical principal polarization, the identity $\beta^2=[-1]$ gives
$(1-\beta)^\vee(1-\beta)=[2]$, whence $\deg(1-\beta)=2$. Supersingularity of $C$
again identifies the kernel with $\alpha_2$.

Jensen's quotient sequence \cite[Theorem~2.1]{Jensen} gives
\(
 0\to G^\vee\to\operatorname{Pic}^0(S)
 \to A^\vee\times J^G\to0.
\)
Here $G^\vee=\mu_p$ (or \(\mu_4\)) is connected, so the quotient sequence restricts to identity components.
Restriction along $A\times\{c_0\}$ compares it with
\( 0\to G^\vee\to B^\vee\to A^\vee\to 0;\)
the maps on the outer terms are the identity and the first projection. Thus
$\ker(s^*)\simeq J^G$. Since $s^*f^*=1$, we obtain
\(\operatorname{Pic}^0(S)\simeq B^\vee\times\alpha_p.\)

Write $\pi:A\times C\to S$ for the quotient map and $f:S\to B=A/G$ for the induced
fibration. We next check that $\operatorname{NS}(S)$ is torsion-free.
If a line bundle $L$ represents a torsion class in $\operatorname{NS}(S)$, then
$\pi^*L$ is algebraically trivial, since $\operatorname{NS}(A\times C)$ is torsion-free.
Its component in $J(C)(k)$ is $G$-invariant and therefore zero, since
$J(C)^G\simeq\alpha_p$. The isogeny $B^\vee\to A^\vee$ is surjective on $k$-points,
so, after tensoring $L$ with the inverse of a line bundle pulled back from
$\operatorname{Pic}^0(B)$, we may assume that $\pi^*L$ is trivial.
The kernel of pullback along $\pi$ is
$\operatorname{Hom}(G,k^\times)=0$, since $G$ has $p$-power order.
Thus $L$ is algebraically trivial, so $\operatorname{NS}(S)$ is torsion-free and hence free
by finite generation.
\end{proof}

Translation acts trivially on rational cohomology, and the $G$-invariant part of
$\mathrm H^1_{\cris}(C/W)[1/p]$ vanishes. Rational crystalline descent and K\"unneth therefore
give the following Betti numbers and slopes:
\begin{equation}\label{eq:quotientslopes}
 \begin{array}{c|ccccc}
  n&0&1&2&3&4\\ \hline
  b_n(S)&1&2&2&2&1\\
  \text{slopes}&0&(0,1)&(1,1)&(1,2)&2.
 \end{array}
\end{equation}
The finite \'{e}tale cover gives $|G|\chi(\mathcal O_S)=\chi(\mathcal O_{A\times C})=0$.
By \ref{Pic(S)}, $h^1(\mathcal O_S)=2$, so $h^2(\mathcal O_S)=1$. 
The long exact sequence shows that \(\HH^2(S,WO_S)/V\cong k\) and
by Ekedahl \cite[III, Proposition~8.1(ii), (8.1.3)]{EkedahlII}, \(0\rightarrow \mathbb{D}\alpha_p^\vee \to \HH^2(S,WO_S)\to \mathrm{TC}(\widehat{\mathrm{Br}}(S))\to 0\).
Counting the dimension of \(\HH^2(S,WO_S)/V\) shows that \(\mathrm H^2(S,W\mathcal O_S)\simeq \mathbb{D}(\alpha_p^\vee)=k\) where \(F=V=0.\)
The only possible domino of a surface is $T^{0,2}$; the finite length of \(\HH^2(S,WO_S)\) forces
this number to vanish. Thus $S$ is Hodge--Witt.

\begin{proposition}\label{prop:quotient-hodge-witt}
Let $S$ be the quotient surface constructed above, over an algebraically closed field $k$,
and put \(M=R\Gamma(S,W\Omega_S^\bullet)\in D_c^b(R).\)
Its Hodge--Witt cohomology is given by
\begin{equation}\label{eq:quotient-hodge-witt-table}
 \begin{array}{c|ccc}
  \mathrm H^j(S,W\Omega_S^i)&i=0&i=1&i=2\\ \hline
  j=2&k&W\oplus k&W\\
  j=1&W&W^{\oplus2}&W\\
  j=0&W&W&0.
 \end{array}
\end{equation}
On each copy of $W$, $F=\sigma$ and $V=p\sigma^{-1}$, and each copy of $k$ has $F=V=0$.
Moreover, there is an isomorphism
\(M\simeq\bigoplus_{j=0}^{2}\mathrm H^j(M)[-j]\) in \(D_c^b(R),\)
where $\mathrm H^j$ denotes cohomology in the external degree.
\end{proposition}

\begin{proof} $S$ is Hodge--Witt, and its slope spectral sequence
degenerates at $E_1$. The rational slope decomposition
\(\mathrm H^j(S,W\Omega_S^i)[1/p]\simeq \mathrm H^{i+j}_{\mathrm{cris}}(S/W)[1/p]_{[i,i+1)}\)
and the slopes in \eqref{eq:quotientslopes} determine the free ranks in
\eqref{eq:quotient-hodge-witt-table}. Thus the torsion-free quotients are
unit-root Dieudonn\'e modules, and are isomorphic to sums of the standard module $W$
because $k$ is algebraically closed. It remains to determine their torsion. 
Ekedahl's duality \cite[IV, Theorem~3.5]{EkedahlI}, in its form for Hodge--Witt surfaces, 
identifies the \(F\)-nilpotent torsion in 
$\mathrm H^2(S,W\Omega_S^1)$ with the Dieudonn\'e dual of the \(F\)-nilpotent torsion part of
$\mathrm H^2(S,W\mathcal O_S)=k$, which is again $k$, with $F=V=0$. 
There is no other torsion: the F-bijective torsion of \(\HH^2(S,W\Omega^1_S)\) 
and \(\HH^1(S,W\Omega^1_S)\) comes from
$\operatorname{NS}(S)[p^\infty]$ and F-bijective torsion of \(\HH^2(S,WO_S)\) 
and \(\HH^1(S,W\Omega^2_S)\) comes from the multiplicative part of
$\operatorname{Pic}^0(S)/\operatorname{Pic}^0(S)_{\mathrm{red}}$, both of which vanish.
The only remaining extension \( 0\to k\to \mathrm H^2(S,W\Omega_S^1)\to W\to0\)
also splits as a Dieudonn\'e module since $F$ kills its torsion and induces an isomorphism
on its torsion-free quotient, so its image maps isomorphically to that quotient
and is stable under both $F$ and $V$. This proves the table.

Regard $W$ and $k$ as graded $R$-modules concentrated in internal degree zero;
then $W(-a)$ and $k(-a)$ are concentrated in internal degree $a$.
The completed resolutions in \cite[III, Lemma~1.5 and Corollary~1.5.4]{EkedahlII}
give, for all integers $a,b$,
\( \operatorname{Ext}^q_R(W(-a),W(-b))=0\) and \(\operatorname{Ext}^q_R(k(-a),W(-b))=0\) 
for every \(q\geqslant2\). The vanishing of these extension groups gives the desired splitting. 
\end{proof}

Let \(\mathcal{E}\rightarrow U\) be any family of elliptic curves with ordinary generic fiber at
\(\eta\) and a supersingular fiber \(E_0\). Form the family \(S\times_k \mathcal{E}\rightarrow U\). 
\begin{theorem}\label{thm:strict-specialization}
The Hodge--Witt numbers of \(X_{\overline{\eta}}\coloneqq (S\times_k \mathcal{E})_{\overline{\eta}}\) and \(X_0\coloneqq (S\times_k \mathcal E)_0\) are 
\begin{equation}\label{eq:threefoldnumbers}
 \begin{array}{c|c|c}
  n&\bigl(h_W^{i,n-i}(X_{\overline\eta})\bigr)_{i=0}^n
   &\bigl(h_W^{i,n-i}(X_0)\bigr)_{i=0}^n\\ \hline
  2&(1,5,1)&(2,3,2)\\
  3&(0,4,4,0)&(1,3,3,1).
 \end{array}
\end{equation}
\end{theorem}

\begin{proof}
  This follows from the K\"unneth formula \cite{EkedahlII} of Ekedahl. Write $E_{1/2}$ for
$\mathrm H^1(E_0,W\mathcal O_{E_0})$, the Dieudonn\'e module of a supersingular elliptic curve.
The only non-trivial calculation is 
\begin{equation}\label{eq:product-basic-domino-calculation}
 \mathrm H^q\bigl(\mathbb D(\alpha_p)\cdstimes E_{1/2}\bigr)\simeq
 \begin{cases}
  U_{-1},&q=-1,\\
  U_1,&q=0,\\
  0,&q\ne-1,0
 \end{cases}
\end{equation}
by \cite[Proposition~3.39]{LiYang}.
\end{proof}

\subsubsection{Ekedahl's Igusa fourfolds}\label{subsubsec:charpIgusa}
The following example is obtained directly from Ekedahl's computation in \cite[III, \S8]{EkedahlII};
the corresponding equal-characteristic family is explicitly discussed on p.~101, Remark~(i). We
record its consequence for Hodge--Witt polygons.

Here $k$ has characteristic $2$, and $\mathcal E\to U$ is a family of elliptic curves
with ordinary generic fiber and a supersingular special fiber. Fix an ordinary elliptic curve
$E_1/k$ and a nonzero point
$a\in E_1[2](k)$. The involution
\[
 (x,y)\longmapsto(x+a,-y)
\]
acts freely on $E_1\times_k\mathcal E$. Its quotient and its relative self-product,
\[
 \mathcal I=(E_1\times_k\mathcal E)/(\mathbb Z/2),\qquad
 \mathcal Y=\mathcal I\times_U\mathcal I,
\]
are smooth projective over $U$, with geometrically connected fibers. In particular, $\mathcal Y$ has
relative dimension $4$.

For either geometric fiber $\mathcal{I}_u$, rational crystalline descent gives
\( \mathrm{H}^1_{\cris}(\mathcal{I}_u/W)[1/2]\simeq \mathrm{H}^1_{\cris}(E_1/W)[1/2],
 \mathrm{H}^2_{\cris}(\mathcal{I}_u/W)[1/2]\simeq W[1/2](-1)^{\oplus2}.\)
Indeed, translation acts trivially on the cohomology of $E_1$, while inversion acts as $-1$ on
$\mathrm{H}^1_{\cris}(\mathcal{E}_u/W)[1/2]$. The K\"unneth formula then gives slopes $0,1,2$ in
$\mathrm{H}^2_{\cris}(\mathcal{Y}_u/W)[1/2]$, with multiplicities $1,6,1$. 
When \(\mathcal{E}_u\) is ordinary, \(\mathcal{Y}_u\) is also ordinary: 
ordinarity descends along the finite \'etale quotient
defining $\mathcal{I}_{u}$, and is preserved by products. If \(\mathcal{E}_u\) is supersingular, 
then Ekedahl's special-fiber calculation \cite[III, (8.10)]{EkedahlII} gives 
\( T^{0,2}(\mathcal{Y}_u)=1,T^{1,1}(\mathcal{Y}_u)=T^{2,0}(\mathcal{Y}_u)=0.\)
Substituting in the definition of Hodge--Witt numbers yields on \(\mathcal{Y}_u\)
\[
 \bigl(h_W^{0,2},h_W^{1,1},h_W^{2,0}\bigr)=
 \begin{cases}
 (1,6,1),&\text{if \(\mathcal{E}_u\) is ordinary},\\
 (2,4,2),&\text{if \(\mathcal{E}_u\) is supersingular}.
 \end{cases}
\]

\subsubsection{Fivefolds with constant Newton polygon}\label{subsubsec:charpGodeauxSerre}
We now use relative Godeaux--Serre approximation and the de~Rham--Witt cohomology of $B\alpha_p$ to
construct a strict specialization with a constant degree-three Newton polygon in every positive
characteristic.

Put \(S=\operatorname{Spec}k[t]_{(t)}\). 
Put \(G=\ker(F-t:\mathbb G_{a,S}\longrightarrow\mathbb G_{a,S})=\operatorname{Spec}\bigl(k[t]_{(t)}[x]/(x^p-tx)\bigr).\)
The group law is induced by addition. Since the defining polynomial is monic, $G$ is finite flat of
rank $p$. Its special fiber is $\alpha_p$. Over the algebraic closure of $k(t)$, choose $a$ with $a^{p-1}=t$. The
substitution $x=ay$ identifies its geometric generic fiber with the constant group $\mathbb
Z/p\mathbb Z$.

Apply the relative Godeaux--Serre construction of \cite[Corollary~3.7]{Kothari} in dimension $4$. It
gives a linear action of $G$ on a projective space, a $G$-stable relative complete intersection $Y$
in the free locus, and a quotient
\(X=Y/G\to S\)
which is smooth projective with geometrically connected fibers. The torsor $Y\to X$ and the
descended equivariant hyperplane bundle give a morphism
\(f:X\to BG\times B\mathbb G_m.\)
If we choose the ambient representation sufficiently large, equivariant weak Lefschetz and the
projective bundle formula show that $f_u$ is a Hodge $4$-equivalence for each geometric point $u$ of
$S$; see \cite[\S4]{Kothari} and \cite[Propositions~5.3, 5.10 and~5.11]{ABM}. 
All the geometric stacks here are quasi-compact, quasi-separated, smooth and Hodge-proper.
Consequently, \cite[Proposition~4.12 and Lemma~4.11]{LiYang} show that $f_u$ is both a de~Rham--Witt
$4$-equivalence and a crystalline $4$-equivalence. In particular, if
\[
 C_u=\cofib\bigl(W\Omega^\bullet(BG_u\times B\mathbb G_m)
                   \to \RG(X_u,W\Omega^\bullet)\bigr),
\]
then \(\mathrm{H}^j(C_u)^i=0\) for \(i+j<4.\)
The long exact sequence gives isomorphisms on $\mathrm{H}^j(W\Omega^i)$ for $i+j<4$, and injections
for $i+j=4$. These are maps of Raynaud modules and are compatible with $d$, $F$ and $V$.

The morphism $f_0:X_0\to B\alpha_p\times B\mathbb G_m$ on the special fiber is a
de Rham--Witt $4$-equivalence. 
We use the computation of $B\alpha_p$ in \cite[Theorem~5.11]{LiYang}, together with the Tate-summand
decomposition for $B\mathbb G_m$ in \cite[Corollary~4.18]{LiYang}. In total degree $3$, the only
nonzero domino of $B\alpha_p\times B\mathbb G_m$ begins in bidegree $(0,3)$, and its dimension is
$1$. The other source groups on that diagonal are finitely generated over $W(k)$. Thus
\(T^{0,3}(X_0)=1\) and the remaining degree-three domino numbers vanish
because their source groups are finitely generated over $W(k)$. 

For the geometric generic fiber \(X_{\overline{\eta}}\), \'etale descent along $\operatorname{Spec}(k)\to \mathbb{Z}/p$ identifies abelian sheaves on
the small \'etale site of $BG$ with $G$-modules, and global sections with $G$-invariants; thus its
de~Rham--Witt cohomology can be computed by group cohomology. 
Since $W\mathcal O$ corresponds to $W(k)$ with trivial $G$-action, passing to right
derived functors gives \(\mathrm R\Gamma(BG,W\mathcal O)\simeq \mathrm R\Gamma(G,W(k)).\) 
In internal degree $0$ this is $\mathrm{H}^a(\mathbb{Z}/p,W)$;
the other internal degrees vanish. The usual periodic resolution has alternating maps $0$ and $p$,
so these groups are finitely generated over $W$. The $B\mathbb G_m$ factor adds Tate summands. Thus
all source groups in total degree $3$ are finite over $W$, and all the corresponding domino numbers
of $X_{\overline\eta}$ vanish.

Crystalline comparison and the rational crystalline computation of the two classifying stacks give,
on both geometric fibers,
\begin{equation}\label{eq:strict-low-crys}
 \mathrm{H}^1_{\cris}(X_u/W)[1/p]=\mathrm{H}^3_{\cris}(X_u/W)[1/p]=0,
 \qquad \mathrm{H}^2_{\cris}(X_u/W)[1/p]\cong W[1/p](-1).
\end{equation}
For $B(\mathbb Z/p)$, the rational assertion also follows by averaging in group cohomology; for $B\alpha_p$,
use \cite[Proposition~4.17]{ABM}. The class in degree $2$ comes from $B\mathbb G_m$.

It follows directly from the definition of the Hodge--Witt numbers that
\begin{align*}
 \bigl(h_W^{0,3},h_W^{1,2},h_W^{2,1},h_W^{3,0}\bigr)(X_0)
     &=(1,-2,1,0),\\
 \bigl(h_W^{0,3},h_W^{1,2},h_W^{2,1},h_W^{3,0}\bigr)(X_{\overline\eta})
     &=(0,0,0,0).
\end{align*}
This already proves variation of the signed numbers in a connected smooth projective family. It does
not prove variation of their supporting-line polygons: the third Betti number is zero, and both
polygons have domain consisting only of the point $0$.

Choose an ordinary elliptic curve $E/k$ and set \(Z=X\times_S(E\times_k S).\)
This is smooth projective of relative dimension $5$ with geometrically connected fibers. 
Since \(R\Gamma(E,W\Omega^\bullet_{E/k})\cong W\oplus W(-1) \oplus W[-1]\oplus W(-1)[-1]\) 
and \(W\) is a product unit for Ekedahl's star product, 
the K\"unneth formula gives the Hodge-Witt numbers of \(Z_0\) and \(Z_{\overline{\eta}}\)

\[
\begin{aligned}
(h_W^{0,3},h_W^{1,2},h_W^{2,1},h_W^{3,0})(Z_{\bar\eta})
&=(0,1,1,0),\\
(h_W^{0,3},h_W^{1,2},h_W^{2,1},h_W^{3,0})(Z_0)
&=(1,-1,2,0).
\end{aligned}
\]
This gives a strict specialization of Hodge--Witt polygons. 

\section{Comparisons in mixed characteristic}\label{sec:mixed}

For a proper smooth variety over a perfect field of characteristic \(p\), the Newton polygon
gives an upper bound for the Hodge polygons of its characteristic-zero liftings. A natural
question is whether the special fiber also provides a lower bound, and whether the Newton upper
bound can be refined. We show that its Hodge--Witt polygon gives a lower bound, while its
abstract Hodge polygon gives a finer upper bound under a torsion-freeness hypothesis.

We now work in the setting of Theorem~\ref{thm:main}. The special-fiber length formula has already
been established in Corollary~\ref{cor:geometricgrowth}. To compare it with the Hodge numbers of the
generic fiber, we construct a prismatic lift of the same finite Nygaard quotients.

Choose compatible primitive roots $\zeta_{p^m}$ and put
$\epsilon=(1,\zeta_p,\zeta_{p^2},\ldots)\in\Oc^\flat$. Set
\[
 \Ainf=W(\Oc^\flat),\qquad \mu=[\epsilon]-1,\qquad
 \xi=\mu/\varphi^{-1}(\mu).
\]
Then $\Ainf/\xi=\Oc$ through $\theta$. Write $\varphi$ for Witt Frobenius and put
\[
 \xi_r=\prod_{j=0}^{r-1}\varphi^{-j}(\xi).
\]
The Witt residue map $\rho:\Ainf \to W(k)$ sends every $\varphi^{-j}(\xi)$ to $p$.

For an $\Ainf$-module or complex $M$, the restriction of scalars $\varphi_*M$ has underlying additive
object $M$ and action $c\cdot m=\varphi(c)m$. In particular, restriction by $\varphi^j$ moves a
complex supported at $\xi$ to one supported at $\varphi^{-j}(\xi)$. 

Our strategy uses the prismatic complex and Andriopoulos's r-Nygaard filtration to construct a
perfect \(\Ainf\)-complex \(\calP^a_{r,\mathfrak{X}/\Oc}\) 
such that \(\calP^a_{r,\mathfrak{X}/\Oc}[1/\xi_r]=0\). For each \(0\leqslant j\leqslant r-1\), 
let \(\varphi^{j}_*B_{\mathrm{dR}}^+\) be the \(\varphi^{-j}(\xi)\)-adic completion of \(\Ainf[1/p]\). 
We will show that the \(\varphi^{-j}(\xi)\)-adic length of \(\HH^n(\calP^a_{r,\mathfrak{X}/\Oc}\otL_{\Ainf} \varphi^{j}_*B_{\mathrm{dR}}^+)\) 
is exactly \(S_a\{h^{i,n-i}\}\). Moreover, the complex \(\calP^a_{r,\mathfrak{X}/\Oc}\) specializes to \(\calP^a_{r,X_k/k}\) of Equation \eqref{calQcalP}. 
An argument using minors then gives the desired inequality. 

\subsection{The Hodge--Witt lower bound}\label{subsec:hwlowerbound}

\begin{notation}Let $\Prism=\Prism_{\X/\Ainf}\coloneqq R\nu_*\mathcal O_\Prism\in D(\X_{\mathrm{\acute et}},\Ainf)$ be the relative
prismatic complex of \cite[Construction~4.9]{BS}, where $\nu$ is the natural projection from the
relative prismatic site to the small \'{e}tale site of $\X$. The comparison \cite[Theorem~17.2]{BS}
identifies $\varphi^*\Prism$ with the complex $A\Omega_\X$ in \cite{BMS}, and the de Rham--Witt 
comparison specializes \(\varphi^*\Prism\) to the de Rham--Witt complex \(W\Omega^\bullet_{X_k/k}\) 
in \(D(X_{k,\mathrm{\acute{e}t}},W)\) \cite[Theorem~14.1(i)]{BMS}. We denote the relative Nygaard
filtration on $\varphi^*\Prism$ by $\Nyg^{\geqslant\bullet}\varphi^*\Prism$ and the associated graded pieces by
$\gr^b_\Nyg\varphi^*\Prism$. Fix $a\geqslant1$. There are canonical maps
\begin{equation}\label{eq:span}
 \begin{aligned}
 \iota: \Nyg^{\geqslant a}\varphi^*\Prism&\longrightarrow\varphi^*\Prism,\\
 \varphi_{\Prism}/\xi^a: \Nyg^{\geqslant a}\varphi^*\Prism&\longrightarrow\Prism,
 \end{aligned}
\end{equation}
where $\varphi_\Prism:\varphi^*\Prism\to\Prism$ is the $\Ainf$-linear relative Frobenius of
\cite[Theorem~15.2]{BS}. Under the canonical identification
$\Prism\simeq\varphi_*\varphi^*\Prism$, multiplication by $\xi$ in the target is multiplication by
$\varphi(\xi)$ on the underlying $\varphi^*\Prism$.

The relative filtered Frobenius and its graded map are \cite[Theorems 15.2--15.3]{BS}
\begin{equation}\label{eq:filteredfrobenius}
 \begin{aligned}
 \varphi_\Prism: \Nyg^{\geqslant\bullet}\varphi^*\Prism&\longrightarrow\xi^\bullet\Prism,\\
 \gr^b_\Nyg\varphi^*\Prism\simeq\tau_{\leqslant b}(\Prism/\xi)\{b\}
       &\longrightarrow(\Prism/\xi)\{b\}.
 \end{aligned}
\end{equation}
The graded map is the canonical truncation inclusion, taken on the sheaf complex before applying
global sections. Braces denote invertible coefficient twists, which preserve the dimensions and
lengths under consideration.
\end{notation}

\begin{notation}
Set $\Nyg^{\geqslant a}_1\varphi^*\Prism=\Nyg^{\geqslant a}\varphi^*\Prism$. With the chosen generator $\xi$ and the
restriction-of-scalars convention above, we use the following relative form of Andriopoulos's
finite $r$-Nygaard construction for $r\geqslant2$:
\begin{equation}\label{eq:Es}
 \Nyg^{\geqslant a}_{r}\varphi^*\Prism=\left(\Nyg^{\geqslant a}
       \times^{h}_{\Prism}\varphi_*\Nyg^{\geqslant a}
       \times^{h}_{\varphi_*\Prism}\cdots
       \times^{h}_{\varphi_*^{r-2}\Prism}\varphi_*^{r-1}\Nyg^{\geqslant a}\right) (\varphi^*\Prism).
\end{equation}
At the $j$-th link, for $0\leqslant j\leqslant r-2$, the two maps to $\varphi_*^j\Prism$ are
$\varphi_*^j(\varphi_{\Prism}/\xi^a)$ from the left and $\varphi_*^{j+1}\iota$ from the right. Set
\[
 \calQ^a_{r,\mathfrak{X}/\Oc}=\cofib(\Nyg^{\geqslant a}_{r}\varphi^*\Prism\longrightarrow\varphi^*\Prism),\qquad
 \calP^a_{r,\mathfrak{X}/\Oc}=\RG(\X,\calQ^a_{r,\mathfrak{X}/\Oc}),
\]
where the map uses projection to the first factor, followed by $\iota$; cf.
\cite[Theorem~1.1, Definition~5.1 and \S5.2]{Andriopoulos}. 
\end{notation}

\begin{lemma}[Filtration by single-step quotients]\label{lem:filtration}
For $r\geqslant2$ there is a distinguished triangle
\begin{equation}\label{eq:QsTriangle}
 \varphi_*^{r-1}\calQ^a_{1,\mathfrak{X}/\Oc}\longrightarrow\calQ^a_{r,\mathfrak{X}/\Oc}
       \longrightarrow\calQ^a_{r-1,\mathfrak{X}/\Oc}\xlongrightarrow{+1}.
\end{equation}
Consequently $\calQ^a_{r,\mathfrak{X}/\Oc}$ has a finite filtration with factors
$\varphi_*^j\calQ^a_{1,\mathfrak{X}/\Oc}$, $0\leqslant j<r$.
\end{lemma}
\begin{proof}
Forgetting the last factor gives the homotopy Cartesian square
\[
 \begin{array}{ccc}
 \Nyg^{\geqslant a}_{r}\varphi^*\Prism
 &\longrightarrow&
 \varphi_*^{r-1}\Nyg^{\geqslant a}
 \\[4pt]
 \big\downarrow
 &&
 \big\downarrow
 \\[4pt]
 \Nyg^{\geqslant a}_{r-1}\varphi^*\Prism
 &\longrightarrow&
 \varphi_*^{r-2}\Prism .
 \end{array}
\]
Since a homotopy Cartesian square in a stable category is also homotopy cocartesian,
the two vertical maps have the same cofiber. Thus
\[
 \cofib(\Nyg^{\geqslant a}_{r}\varphi^*\Prism\longrightarrow \Nyg^{\geqslant a}_{r-1}\varphi^*\Prism)
 \simeq\varphi_*^{r-1}\calQ^a_{1,\mathfrak{X}/\Oc}.
\]
The cofiber triangle for the composite
$\Nyg^{\geqslant a}_{r}\varphi^*\Prism\to \Nyg^{\geqslant a}_{r-1}\varphi^*\Prism\to\varphi^*\Prism$ now gives
\eqref{eq:QsTriangle}, and induction gives the filtration.
Applying derived global sections preserves these triangles.
\end{proof}

The residue map defines a morphism of perfect prisms $(\Ainf,\xi)\to(W,p)$, where $W=W(k)$ and
$\xi$ is mapped to \(p\). By the crystalline comparison
\cite[Theorem~5.2]{BS} and the de~Rham--Witt comparison, the complexes $W\Omega^\bullet_{X_k/k}$ and
$\sigma_*W\Omega^\bullet_{X_k/k}$ represent, respectively, the derived crystalline specializations of
$\varphi^*\Prism$ and $\Prism$ in $D(X_{k,\mathrm{\acute et}},W)$.

The relative Nygaard filtration commutes with completed derived base change \cite[Theorems~15.2(3)
and~15.3]{BS}. Under the crystalline comparison, its specialization is represented by the classical
filtration of subcomplexes in \eqref{eq:rthnygaard} for $r=1$ \cite[Proposition~8.5 and
Corollary~8.6]{BMS2}:
\[
 \Nyg^{\geqslant a}\varphi^*\Prism\widehat\otimes^{\mathrm{L}}_{\Ainf} W
 \simeq\Nyg^{\geqslant a}W\Omega^\bullet_{X_k/k}.
\]
This comparison respects the inclusion and divided Frobenius. On graded pieces, it matches their
descriptions as $\tau_{\leqslant b}(\sigma_*\Omega^\bullet_{X_k/k})$ in filtration degree $b$; see
\eqref{eq:filteredfrobenius} and \cite[Lemma~8.2]{BMS2}.

\begin{proposition}[Crystalline comparison]\label{prop:crystallinequotient}
For every $r,a\geqslant1$, there is a natural equivalence
\begin{equation}\label{eq:specialcomparison}
 \calP^a_{r,\mathfrak{X}/\Oc} \otL_{\Ainf}W\simeq \calP^a_{r,X_k/k}.
\end{equation}
\end{proposition}
\begin{proof}
The preceding crystalline comparison identifies the sheaf-level specialization of $\calQ^a_{r,\mathfrak{X}/\Oc}$ with
the displayed quotient. The base-change argument in the proof of Lemma~\ref{lem:perfectness}
permits passage to derived global sections.
\end{proof}

\begin{lemma}[Perfectness and support]\label{lem:perfectness}
For $r,a\geqslant1$, the complex $\calP^a_{r,\mathfrak{X}/\Oc}$ is perfect over $\Ainf$ and satisfies
\begin{equation}\label{eq:support}
 \calP^a_{r,\mathfrak{X}/\Oc}[1/\xi_r]=0.
\end{equation}
\end{lemma}
\begin{proof}
First, $\RG(\X,\Omega^i_{\X/\Oc})$ is perfect over $\Oc$ and satisfies derived base change to the
special fiber; here the forms are continuous. Indeed, put $\X_m=\X\bmod p^m$. Proper perfect
pushforward and derived base change give compatible perfect complexes $\RG(\X_m,\Omega^i)$ over
$\Oc/p^m$ \cite[Tag~0DJT]{Stacks}. Their derived inverse limit is perfect over $\Oc$ and recovers
each reduction \cite[Tag~09AW]{Stacks}.

This inverse limit computes formal coherent cohomology. A finite affine formal cover with affine
intersections gives \v{C}ech complexes on the reductions with termwise surjective transition maps,
so their derived inverse limit is the formal \v{C}ech complex. Derived base change to $k$ follows
from that for $\X_1$ over $\Oc/p$.

By \eqref{eq:filteredfrobenius}, the finite Nygaard filtration on $\calQ^a_1$ has graded objects
$\tau_{\leqslant b}(\Prism/\xi)\{b\}$ for $0\leqslant b<a$. Taking their finite Postnikov
filtrations gives pieces $\Omega^i_{\X/\Oc}[-i]$ with invertible twists, for $0\leqslant i\leqslant
b$. Their global sections are perfect over $\Oc$, hence over $\Ainf$, since $\xi$ is a nonzerodivisor.
This proves perfectness for $r=1$. Lemma~\ref{lem:filtration} proves it for all $r$ and gives
\eqref{eq:support}, since its $j$-th factor is supported at $\varphi^{-j}(\xi)$.

For $\calQ^a_1$, crystalline specialization commutes with global sections by coherent derived base
change on these finite filtrations and the identity
\((\Ainf/\xi)\otL_{\Ainf} W\simeq[W\xrightarrow{p}W]\simeq k.\)
The filtration in Lemma~\ref{lem:filtration} gives the same assertion for every $r$. The resulting
global complexes are perfect over $W$, so ordinary and completed derived tensor products agree.
\end{proof} 

The ring $\varphi_*^j\BdR$  can be understood as the $\varphi^{-j}(\xi)$-adic completion of $\Ainf[1/p]$. It is a complete discrete 
valuation ring with uniformizer $\varphi^{-j}(\xi)$
and residue field $\Ainf[1/p]/(\varphi^{-j}(\xi))$, identified with $C$ by $\theta\varphi^j$. The automorphism
$\varphi^j$ identifies $\varphi_*^j\BdR$ with $\BdR$.

\begin{lemma}[Specialization of lengths]\label{lem:determinant}
Let $\calP$ be a perfect $\Ainf$-complex such that $\calP\otL_{\Ainf}W[1/p]$ is acyclic. Then for every $n$, 
\(\ell_{\varphi^j_*\BdR}^n(\calP\otL_{\Ainf} \varphi^j_*\BdR)\) is finite for all \(j\in \mathbb Z\) and non-zero 
for only finitely many \(j\). Moreover, 
\begin{equation}\label{eq:detbound}
 \ell^n(\calP\otL_{\Ainf} W) \geqslant\ \sum_{j\in \mathbb Z}\ell_{\varphi^j_*\BdR}^n(\calP\otL_{\Ainf} \varphi^j_*\BdR).
\end{equation}
\end{lemma}
\begin{proof}
The ring $\Ainf$ is local. Choose a bounded finite free representative $C^\bullet$ of $\calP$. Write $r_q$ and $t_q$ for the ranks of $dC^q\subset C^{q+1}$ after base change to $\Frac(\Ainf)$
and $W[1/p]$, respectively. The rank of each differential is the largest size of a nonzero minor.
A minor nonzero after applying $\rho$ was already nonzero in $\Ainf$, so $r_q\geqslant t_q$. By hypothesis,
$C^\bullet\otimes_{\Ainf}W[1/p]$ is acyclic, and hence
\[
 \operatorname{rank}C^q
 =t_{q-1}+t_q
 \leqslant r_{q-1}+r_q
 \leqslant\operatorname{rank}C^q.
\]
Thus $r_q=t_q$ for every $q$, and the complex is also acyclic over
$\Frac(\Ainf)$. Since $\Ainf$ embeds in each $\varphi^j_*\BdR$, the differential
ranks over $\Frac(\varphi^j_*\BdR)$ are the same. Consequently, the cohomology
groups over $W$ and $\varphi^j_*\BdR$ have finite length. 

Put $t=r_{n-1}$. Lemma~\ref{lem:minors} computes the cohomology lengths from the maximal nonzero minors of $d:C^{n-1}\rightarrow C^n$.
If $t=0$, all these lengths vanish, so assume $t>0$.

Put $\ell_j=\ell_{\varphi^j_*\BdR}^n(\calP\otL_{\Ainf} \varphi^j_*\BdR)$. Every $t$-minor $f\in \Ainf$ lies in
$\varphi^{-j}(\xi)^{\ell_j}\varphi^j_*\BdR$. We next contract this divisibility to $\Ainf$. For every $m\geqslant1$,
\[
 \Ainf[1/p]/(\varphi^{-j}(\xi)^m)\simeq \varphi^j_*\BdR/(\varphi^{-j}(\xi)^m).
\]
It follows that $f\in\varphi^{-j}(\xi)^{\ell_j}\Ainf[1/p]$; when $\ell_j=0$ there is no condition to contract. Since
$\Ainf/(\varphi^{-j}(\xi))\simeq\Oc$ is a domain, $\varphi^{-j}(\xi)$ is prime; it does not divide $p$. Successive prime-factor
cancellation yields
\[
 \varphi^{-j}(\xi)^m\Ainf[1/p]\cap \Ainf=\varphi^{-j}(\xi)^m\Ainf.
\]
The $\varphi^{-j}(\xi)$ are pairwise nonassociated prime elements. Successive cancellation therefore shows that
divisibility by every $\varphi^{-j}(\xi)^{\ell_j}$ implies divisibility by their 
product over any finite subset \(J\) of \(\mathbb Z\). 

Consequently every $t$-minor is divisible in $\Ainf$ by $\prod_{j\in J}\varphi^{-j}(\xi)^{\ell_j}$. After $\rho$, every
specialized minor is divisible by $p^{\sum_{j\in J}\ell_j}$. At least one such minor is nonzero, since the
specialized differential still has rank $t$. Formula \eqref{eq:minorlength} over $W$ proves\[
 \sum_{j\in J}\ell_j
 \leqslant\ell^n(\calP\otL_{\Ainf}W).
\]
Since this holds for every finite $J$, only finitely many $\ell_j$ are nonzero,
and \eqref{eq:detbound} follows.
\end{proof}

\begin{remark}
This argument is inspired by the torsion comparison in \cite[Lemma~4.14, Corollary~4.15 and
Lemma~4.16]{BMS}. Those results control vertical $p$-power torsion. For the complexes $\calP^a_r$
used below, rational cohomology is supported on the horizontal divisors $\varphi^{-j}(\xi)$, and
Lemma~\ref{lem:determinant} compares individual
cohomology lengths by adding their contributions to the valuation of a maximal nonzero minor.
\end{remark}

\begin{proposition}\label{prop:genericlength}
For $r,a\geqslant1$ and integers $n,j$, we have
\begin{equation}
\label{eq:genericlength}\ell_{\varphi^j_*\BdR}^n(\calP^a_{r,\mathfrak{X}/\Oc}\otL_{\Ainf} \varphi^j_*\BdR)=\begin{cases}
  S_a\{h^{i,n-i}_{\mathfrak{X}_\eta}\} \;&\text{if}\; 0\leqslant j<r\\
  0 \;&\text{otherwise.}
\end{cases}  
\end{equation}
\end{proposition}
We use coherent generic-fiber comparison
\( \RG(\X,\Omega^i_{\X/\Oc})[1/p]
 \simeq\RG(X_\eta,\Omega^i_{X_\eta/C}).\) \cite[\S1.15, Theorem~5 and Corollary~7; \S2.4, Proposition~3]{BoschFormalRigid}. 

\begin{lemma}\label{lem:spectral}
Let $\mathrm{Fil}^\bullet A\to \mathrm{Fil}^\bullet B$ be a map of finitely filtered complexes in an abelian category.
Suppose that, for some $r\geqslant1$, the induced map on $E_r$-pages is injective in every
bidegree and the spectral sequence of $\mathrm{Fil}^\bullet B$ degenerates at $E_r$. Then the
spectral sequence of $\mathrm{Fil}^\bullet A$ also degenerates at $E_r$.
\end{lemma}
\begin{proof}
The map of $E_r$-pages is injective, and the target differential vanishes. Hence the source
differential also vanishes. The map remains injective on the next page, and induction gives the
same assertion on every subsequent page.
\end{proof}

\begin{proof}[Proof of Proposition~\ref{prop:genericlength}]
We first treat $r=1$ and $j=0$. Apply the filtered Frobenius \(\varphi_\Prism\) \eqref{eq:filteredfrobenius}, truncated
at $a$, to obtain
\begin{equation}\label{eq:filteredmap}
 \calP^a_{1,\mathfrak{X}/\Oc}\longrightarrow \RG(\X,\Prism)\otL_{\Ainf} \Ainf/\xi^a.
\end{equation}
After base change to $\BdR$, its grade $b$, for $0\leqslant b<a$, is the natural map
\[
 \RG\bigl(\X,\tau_{\leqslant b}(\Prism/\xi)\bigr)[1/p]\{b\}
 \longrightarrow
 \RG\bigl(\X,\Prism/\xi\bigr)[1/p]\{b\}.
\]
After inverting $p$, the Postnikov spectral sequence of the target identifies with the Hodge--Tate
spectral sequence of $X_\eta$ \cite[Theorem~9.2(i)]{BMS}, which degenerates by
\cite[Theorem~1.7(ii)]{BMS}.

In Postnikov degree $i$, the map above induces the identity on graded pieces for $i\leqslant
b$ and has zero source for $i>b$. It is therefore injective on the cohomology of every graded
piece. Since the target spectral sequence degenerates, Lemma~\ref{lem:spectral} implies that the
source spectral sequence also degenerates. Consequently, its degree-$n$ cohomology has dimension
\( \sum_{0\leqslant i\leqslant b}h^{i,n-i}.\)
The induced map on $E_\infty$ is injective, so the map on cohomology is injective as well. 

After inverting \(p\), the spectral sequence of the finite $\xi$-adic filtration on 
\(\RG(\X,\Prism)\otL_{\Ainf} \Ainf/\xi^a\) degenerates at $E_1$. To see this, put $C_0=\RG(\X,\Prism)\otL_{\Ainf} \BdR$. By
\cite[Theorem~1.8]{BMS} and Frobenius pullback, $\RG(\X,\Prism)[1/p]$ has finite free cohomology
over $\Ainf[1/p]$. It therefore splits noncanonically as the direct sum of its cohomology modules in
their degrees. Base change gives the same splitting for $C_0$, with finite free $\BdR$-cohomology.
The splitting induces a splitting of the filtered complex $C_0\otL_{\BdR}\BdR/\xi^a$ with its finite
$\xi$-adic filtration, proving the claimed degeneration. 

Applying Lemma~\ref{lem:spectral} to \(\calP^a_{1,\mathfrak{X}/\Oc}\rightarrow \RG(\X,\Prism)\otL_{\Ainf} \Ainf/\xi^a\) after inverting \(p\) now gives the exact length
\begin{equation}\label{eq:ordinarylength}
 \ell_{\BdR}^n(\calP^a_{1,\mathfrak{X}/\Oc}\otL_{\Ainf} \BdR)
 =\sum_{b=0}^{a-1}\sum_{i=0}^b h^{i,n-i}
 =S_a\{h^{i,n-i}\}.
\end{equation}
For $\ell\ne j$, the element $\varphi^{-\ell}(\xi)$ is a unit in $\varphi^j_*\BdR$. Indeed its image under $\theta\varphi^j$
is nonzero; the cyclotomic formula gives $p$ for a positive Frobenius shift and a nonzero ratio of
successive cyclotomic differences for a negative shift. If $0\leqslant j<r$, the finite filtration
of Lemma~\ref{lem:filtration} has exactly one surviving factor after tensoring with
$\varphi^j_*\BdR$, namely $\varphi_*^j\calP^a_1$; otherwise all factors vanish. Frobenius
transports the length of the surviving factor to \eqref{eq:ordinarylength}, proving
\eqref{eq:genericlength}.
\end{proof}

By Lemma~\ref{lem:perfectness}, we may apply Lemma~\ref{lem:determinant} to $\calP^a_r$. Using
\eqref{eq:specialcomparison} and \eqref{eq:genericlength}, we obtain, for every
$r\geqslant1$,
\begin{equation}\label{eq:finiteinequality}
 \ell^n\bigl(\RG\bigl(X_k,W\Omega^\bullet/\Nyg^{\geqslant a}_r W\Omega^\bullet\bigr)\bigr)
       \geqslant rS_a\{h^{i,n-i}\}.
\end{equation}

\begin{proof}[Proof of Theorem~\ref{thm:main}]
Combining \eqref{eq:finiteinequality} and \eqref{eq:geometricgrowth} gives
\[
 rS_a\{h^{i,n-i}\}\leqslant rS_a\{h_W^{i,n-i}\}+O(1)
\]
for every fixed $a\geqslant1$ and all $r\geqslant1$. Divide by $r$ and let $r\to\infty$. This proves
the inequalities in Lemma~\ref{lem:support} for $a\geqslant1$. For $a\leqslant0$ both weighted sums vanish, since the geometric
sequences have no negative internal indices. Lemma~\ref{lem:support} gives the polygon inequality. 
The generic Hodge polygon has the same endpoint as the Newton and Hodge--Witt polygons of the
special fiber.
\end{proof}

By Lemma~\ref{lem:telescope}, the inequality can also be written as
\begin{equation}\label{eq:dominodefect}
 \sum_{i<a}(a-i)h^{i,n-i}(X_\eta)
 -\sum_{\lambda<a}(a-\lambda)m_{\lambda}
 \leqslant T^{a-1,n-a+1}(X_k).
\end{equation}
Thus the domino at the cut bounds the difference between the Hodge and Newton weighted sums,
including when some Hodge--Witt numbers are negative.

\begin{remark}[The first cut]
For $a=1$, the special quotient is $W_r\mathcal O_{X_k}$ in degree zero. The finite-level inequality
becomes
\( \ell^n(\RG(X_k,W_r\mathcal O_{X_k})) \geqslant r\,\dim_C \mathrm{H}^n(X_\eta,\mathcal O_{X_\eta}),\)
and its limit gives
\( h^{0,n}(X_\eta)\leqslant h_W^{0,n}(X_k)=m^{0,n}(X_k)+T^{0,n}(X_k).\)
\end{remark}

\subsection{The abstract Hodge upper bound}\label{subsec:abstracthodge}
The Newton polygon need not give a sharp upper bound for the Hodge polygon of the generic fiber.
Indeed, when the special fiber is Mazur--Ogus, \(\Hdg_n(X_\eta)=\Hdg_n(X_k)\). 
It is natural to ask whether the Newton upper bound can be refined. A
candidate is a bound given by the abstract Hodge polygon. More precisely, for a proper smooth
variety \(X_k/k\), 
let $\HH=\HH^n_{\cris}(X_k/W)/\mathrm{tors}$, and write
\( 0\leqslant\alpha_1\leqslant\cdots\leqslant\alpha_{b_n}\)
for the elementary-divisor exponents of the linearized crystalline Frobenius $\sigma^*\HH\to \HH$. 
Explicitly, \[\HH/F\HH\cong \prod_{i=1}^{b_n} W/p^{\alpha_i}.\]
The \emph{abstract Hodge polygon} $\Hdg_n^{\mathrm{abs}}(X_k)$ in degree $n$ has slopes
$\alpha_1,\ldots,\alpha_{b_n}$, each with horizontal length one. Its endpoint is $(b_n,t_{N,n})$,
and the Newton--Hodge inequality for $F$-crystals \cite[\S1.4]{KatzSlope} gives
\[
 \operatorname{Newt}_n(X_k)\geqslant\Hdg_n^{\mathrm{abs}}(X_k).
\]
See also \cite[IV, Definition~3.4(iii) and Theorem~3.5]{Ekedahl}. The abstract Hodge numbers also have the following interpretation: they are the Hodge and Hodge--Witt numbers of 
the canonical Mazur--Ogus subquotient of \(\Ht^n(\RG(X_k,W\Omega^\bullet_{X_k/k}))\in \dc\) 
in the sense of Ekedahl's two fundamental filtrations
\cite[I, Theorem~4.5(iv); IV, proof of Theorem~3.5]{Ekedahl}.

We spell out this interpretation using Ekedahl's two fundamental filtrations. 
For each object \(M\in \dc\), Ekedahl defined a maximal subobject \(F^1(M)\subseteq M\) whose total
cohomology is concentrated in degree zero and whose \(\mathrm{H}^0\) is torsion. 
Next, there is a maximal quotient \(M/F^2(M)\) of \(M/F^1(M)\) whose non-zero total cohomology is in degree one.
Similarly, there is a maximal quotient \(M/G^2(M)\) whose non-zero total cohomology is in degree one, 
and \(G^1(M)\) is the maximal subobject of \(G^2(M)\) such that its only non-zero total cohomology 
is \(\mathrm{H}^0\) and it is torsion. Consider
\(M\coloneqq \Ht^n\RG(X_k,W\Omega^\bullet_{X_k/k})\in\dc\). Then Ekedahl's fundamental filtrations give a canonical subquotient
\[
 M^{\mathrm{MO}}\coloneqq F^2(M)/F^1(M)\simeq G^2(M)/G^1(M).
\]
The middle subquotient $M^{\mathrm{MO}}$ is Mazur--Ogus, and its zeroth total cohomology is
the Frobenius lattice $\HH$ \cite[I, Theorem~4.5 and Corollary~4.5.3]{Ekedahl}. Thus, writing
$h_{\mathrm{abs}}^{i,n-i}(X_k)=\#\{j:\alpha_j=i\}$, we have
\begin{equation}\label{eq:abstractmo}
 h_{\mathrm{abs}}^{i,n-i}(X_k)
 =h^{i,-i}(M^{\mathrm{MO}})=h_W^{i,-i}(M^{\mathrm{MO}}),
\end{equation}
as in \cite[proof of IV, Theorem~3.5]{Ekedahl}.
The difference from the Hodge--Witt numbers of $X_k$ is measured by the additional domino
contributions outside this subquotient. More precisely, put
\[
 e_i=T^{i,-i}(M)-T^{i,-i}(M^{\mathrm{MO}})\geqslant0.
\]
Here the inequality follows from \cite[III, (2.6), as used in the proof of IV, Theorem~3.5]{Ekedahl}.
The slope numbers of $M$ and $M^{\mathrm{MO}}$ agree, so we have
\begin{equation}\label{eq:extradominodefect}
 h_W^{i,n-i}(X_k)-h_{\mathrm{abs}}^{i,n-i}(X_k)
 =e_i-2e_{i-1}+e_{i-2}.
\end{equation}
These are the extra domino contributions; the Mazur--Ogus subquotient itself may have nonzero
domino numbers. This interpretation does not require crystalline cohomology to be torsion-free.
Ekedahl also proved the following chain of inequalities in full generality:
\[\mathrm{Newt}_n(X_k)\geqslant \Hdg_n^{\mathrm{abs}}(X_k) \geqslant \HW_n(X_k) \geqslant \Hdg_n(X_k).\]
This suggests the following conjecture. 
\begin{conjecture}
  In the setting of Theorem~\ref{thm:main},
  \(\Hdg_n^{\mathrm{abs}}(X_k) \geqslant\Hdg_n(X_\eta)\).
\end{conjecture}
We have not proved the conjecture in full generality and do not know whether it is true. 
The conjecture is also motivated by the description of the generic fiber's Hodge polygon in terms of
elementary divisors. The following description follows from Wu's reconstruction of the
Hodge--Tate filtration \cite{WuHT}.

\begin{lemma}\label{lem:infinityhodge}
Let $\mathfrak X$ be a proper smooth formal scheme over $\mathcal O_C$. Then we have

\[\HH^n\RG(\X,\Prism)[1/p]/\varphi_{\Prism}\cong \prod_{i=1}^{n}(\BdR/\xi^{i})^{h_{X_\eta}^{i,n-i}}\]
\end{lemma}

\begin{proof}The argument relies mainly on \cite[Theorem~1.1]{WuHT}. Recall that $\BdR$ is the $\xi$-adic
completion of $\Ainf[1/p]$. Write $\xi^{a_1},\ldots,\xi^{a_{b_n}}$ for the elementary divisors
of the linearized Frobenius over $\BdR$. Let's set \(h_{\inf}^{i,n-i}(M)=\#\{j:a_j=i\}.\)
All these constructions depend only on \(\HH^n\RG(\X,\Prism)[1/p]\), so we may replace $\HH^n\RG(\X,\Prism)$ by
the finite free envelope of its torsion-free quotient \cite[Proposition~4.13]{BMS}.

Consider \(V\coloneqq \mathrm H^n_{\mathrm{\acute et}}(X_\eta,\mathbb Q_p)\), \(\Lambda=V\otimes_{\mathbb Q_p}B_{\mathrm{dR}}^+\) 
and \(\Xi=\mathrm H^n_{\mathrm{cris}}(X_\eta/B_{\mathrm{dR}}^+).\)
The comparisons in \cite[Proposition~9.17 and its proof; proofs of Theorems~13.1 and~14.5(i)]{BMS} identify the
linearized Frobenius with the inclusion $\Xi\hookrightarrow\Lambda$
\cite[Theorems~1.8 and~14.5(i)]{BMS}.
Consequently, there is a basis $e_1,\ldots,e_{b_n}$ of $\Lambda$ such that
\[\Xi=\bigoplus_{j=1}^{b_n}\xi^{a_j}B_{\mathrm{dR}}^+e_j.\]

Consider the increasing filtration on $\Lambda/\xi\Lambda$ given by
\( F_m=\operatorname{im}\!\left( \xi^{-m}\Xi\cap\Lambda\to\Lambda/\xi\Lambda \right).\)
The displayed decomposition gives
\(\dim_C(F_m/F_{m-1})=\#\{j:a_j=m\}.\)
By \cite[Theorem~1.1]{WuHT}, this is the Hodge--Tate filtration.
Its graded pieces are
$\HH^{n-m}(X_\eta,\Omega^m_{X_\eta/C})(-m)$, so comparison of dimensions
proves the assertion.
\end{proof}

The matrix-specialization argument below is closely related to the proof of
\cite[Proposition~2.5]{Yao}; the integral crystalline comparison allows us to retain
the abstract Hodge polygon before passing to Newton slopes.

\begin{proof}[Proof of Theorem~\ref{thm:abstracthodge}]
Put $M=\HH^n\RG(\X,\Prism)$. Applying \cite[Theorem~14.5(iii)]{BMS} to $\varphi^*M$ and
then restricting scalars along $\varphi$, we see that the first torsion-freeness hypothesis makes
$M$ finite free over $\Ainf$, and the second gives a Frobenius-compatible identification
\[
 M\otimes_{\Ainf} W\simeq \sigma_*\mathrm{H}^n_{\cris}(X_k/W).
\]
As before, $\BdR$, the $\xi$-adic completion of $\Ainf[1/p]$, is a complete discrete valuation ring
with uniformizer $\xi$. The linearized Frobenius $F_M:\varphi^*M\to M$ becomes an
isomorphism after inverting $\xi$. Let $0\leqslant\beta_1\leqslant\cdots\leqslant\beta_{b_n}$ be its
elementary-divisor exponents over $\BdR$.

Choose bases and let $G$ be the matrix of $F_M$. For $1\leqslant\ell\leqslant b_n$, every
$\ell$-by-$\ell$ minor of $G$ belongs to $\xi^{\beta_1+\cdots+\beta_\ell}\BdR$. The contraction
argument in the proof of Lemma~\ref{lem:determinant}, applied to $\xi$, gives
\[
 \xi^e\BdR\cap \Ainf=\xi^e\Ainf\qquad(e\geqslant0).
\]
Since $\rho(\xi)=p$, every specialized minor is divisible by $p^{\beta_1+\cdots+\beta_\ell}$. The
elementary-divisor formula over $W$ therefore gives
\(
 \sum_{j=1}^{\ell}\alpha_j\ \geqslant\
 \sum_{j=1}^{\ell}\beta_j\qquad(1\leqslant\ell\leqslant b_n).
\) By Lemma~\ref{lem:infinityhodge}, the $\beta_j$ are exactly the Hodge slopes of
$X_\eta$, proving the asserted inequality.
Since $\xi$ is prime and $G$ is invertible over $\Ainf[1/\xi]$, we have
$\det G=u\xi^e$ for some $u\in\Ainf^\times$, where $e=\sum_j\beta_j$.
Applying $\rho$ gives
\[
 \sum_j\alpha_j=v_p(\det\rho(G))=e=\sum_j\beta_j.
\]
Hence the two polygons have common endpoints.
\end{proof}

\begin{remark}\label{rem:nonfreebkf}
In general, for a nonfree \(\Ainf\) module $M$ equipped with a semilinear \(\varphi_M\) such that \((M/\varphi_M(M))[1/\xi]=0\), 
the abstract Hodge polygon of its crystalline specialization need not lie above 
the Hodge polygon formed by \(\xi\)-adic elementary divisors. There is a concrete counterexample. Put 
\( M=(p,\varphi^{-1}(\xi))e_1\oplus \Ainf e_2\subset \Ainf e_1\oplus \Ainf e_2.\)
The regular sequence $(p,\varphi^{-1}\xi)$ shows that $M$ is finitely presented and torsion-free, but not free.
It nevertheless satisfies $M[1/p]=\Ainf[1/p]^{\oplus2}$. Restrict to $M$ the $\varphi$-semilinear
Frobenius given by
\(\varphi_M(e_1)=\xi e_2\), \(\varphi_M(e_2)=\xi e_1.\)
Since $\xi\in(p,\varphi^{-1}(\xi))$, this map preserves $M$. After inverting $\xi$, 
both $(p,\varphi^{-1}(\xi))$ and $(p,\xi)$ become
the unit ideal, so \((M/\varphi_M(M))[1/\xi]=0\). Its elementary divisors over $\BdR$ are
$(1,1)$.

Crystalline specialization gives \( (M\otimes_{\Ainf} W)/\mathrm{tors}\simeq pWe_1\oplus We_2\).
In the basis $pe_1,e_2$ of the latter lattice, Frobenius has matrix
\(
 \begin{pmatrix}0&1\\p^2&0\end{pmatrix}.
\)
Its abstract Hodge polygon therefore has slopes $(0,2)$ and lies strictly below the generic Hodge
polygon between their common endpoints. This shows that the comparison can fail algebraically without the freeness assumption. We have not
found an explicit geometric counterexample to Theorem~\ref{thm:abstracthodge} with its crystalline
torsion-freeness hypotheses removed.
\end{remark}

\begin{corollary}[Determination of Hodge numbers]\label{cor:determinedhodge}
Assume that $\mathrm{H}^n_{\cris}(X_k/W)$ and $\mathrm{H}^{n+1}_{\cris}(X_k/W)$ are torsion-free,
and suppose that the diagonal cohomology object
\[\Ht^n\RG(X_k,W\Omega^\bullet_{X_k/k})\in\dc\]
is Mazur--Ogus in the sense of \cite[I, Definition~4.3]{Ekedahl}. Then
\(h^{i,n-i}(X_\eta)=h_W^{i,n-i}(X_k)\) for every \(i\).
In particular, the degree-$n$ Hodge numbers are the same for all proper smooth liftings of $X_k$
covered by Theorem~\ref{thm:main}.
\end{corollary}
\begin{proof}
Since $\Ht^n\RG(X_k,W\Omega^\bullet_{X_k/k})$ is Mazur--Ogus, its Mazur--Ogus subquotient is itself.
Therefore \(\Hdg^{\mathrm{abs}}_n(X_k)=\HW_n(X_k).\)
Theorems~\ref{thm:main} and~\ref{thm:abstracthodge} force $\Hdg_n(X_\eta)$ to equal this common
polygon. The Hodge--Witt numbers of a Mazur--Ogus object are nonnegative \cite[IV,
Corollary~3.3.1]{Ekedahl}, so equality of polygons gives equality of their slope multiplicities.
\end{proof}

\subsection{Examples in mixed characteristic}\label{subsec:mixedexample}

It also seems common for the generic Hodge polygon to lie strictly above the Hodge--Witt
polygon of the special fiber. When the special fiber is not Mazur--Ogus, taking its product
with a supersingular or non-Mazur--Ogus variety can create additional dominos, which account for
such a gap. Nevertheless, the lower bound in Theorem~\ref{thm:main} can be attained even when
the special fiber has many extra dominos. We will also see two examples in which liftings
realize all possible generic-fiber Hodge polygons between the Newton and Hodge--Witt polygons.

We analyze four examples in mixed characteristic. The first is the product of an Enriques surface
and an elliptic curve with supersingular reduction. The second is obtained by lifting the
characteristic-\(p\) surface \(S\) to \(\mathcal{S}\) and taking its product with an elliptic curve
with supersingular reduction. 
The third example is the Godeaux--Serre approximation of a group scheme, multiplied by an ordinary curve of genus two.

Finally, we analyze Li's example in detail. We compute the Hodge numbers of a collection of \(p-1\)
liftings of his special fiber \(Z\) and show that their Hodge polygons attain both its Newton
polygon and its Hodge--Witt polygon. 
In particular, when \(p=5\) or \(7\), these liftings realize every possible generic-fiber Hodge
polygon. 

\subsubsection{A supersingular Enriques surface}\label{subsubsec:mixedEnriques}
Let $k$ be algebraically closed of characteristic \(2\), and let $T/k$ be a
supersingular Enriques surface, so that \(\operatorname{Pic}^{\tau}_{T/k}\simeq\alpha_2.\)
By \cite[Theorem~4.10]{LiedtkeEnriquesLift}, $T$ admits a smooth projective
lifting $\mathcal T$ over a finite, possibly ramified extension $\mathcal O_K$
of $W(k)$.
Choose a supersingular elliptic curve $E_0/k$ and a smooth projective lifting
$\mathcal E/\mathcal O_K$, enlarging $\mathcal O_K$ if necessary, and put
\(\mathcal Z=\mathcal T\times_{\mathcal O_K}\mathcal E.\) 
The generic fiber is the product of a characteristic-zero Enriques surface
and an elliptic curve. Hence
\[
 \bigl(h^{i,2-i}(Z_\eta)\bigr)_i=(0,11,0),\qquad
 \bigl(h^{i,3-i}(Z_\eta)\bigr)_i=(0,10,10,0).
\]

For the Hodge--Witt numbers, write $M=R\Gamma(T,W\Omega_T^\bullet)$.
The de Rham--Witt calculation for a supersingular Enriques surface gives
\(\mathrm H^0(M)\simeq W\), \(\mathrm H^1(M)\simeq W(-1)^{\oplus10}\),
\(\mathrm H^2(M)\simeq k\oplus k(-1)\oplus W(-2)\)
with zero internal differentials; see, for example,
\cite[II, Proposition~7.3.6(b)(iii)]{IllusieDRWCrystalline}. 
Here the two copies of $k$ have $F=V=0$.
The Ext vanishing used in Proposition~\ref{prop:quotient-hodge-witt} again
splits $M$ into its external cohomology modules. 
By the completed K\"unneth formula and
\cite[Proposition~3.39]{LiYang}, the summands $k$ and $k(-1)$ produce
$U_{-1}$ and $U_{-1}(-1)$ in external degree two, and
$U_1$ and $U_1(-1)$ in external degree three.
Thus \(T^{0,2}(Z_0)=T^{1,2}(Z_0)=T^{0,3}(Z_0)=T^{1,3}(Z_0)=1,\)
and all other domino numbers vanish.
The degree-two Newton slopes are all one, with multiplicity eleven,
and the degree-three Newton slopes are all $3/2$, with multiplicity twenty.
The Hodge--Witt formula gives the special-fiber Hodge--Witt numbers
\begin{equation}\label{eq:enriques-mixed-example}
 \bigl(h_W^{i,2-i}(Z_0)\bigr)_i=(1,9,1),\qquad
 \bigl(h_W^{i,3-i}(Z_0)\bigr)_i=(1,9,9,1).
\end{equation}

\subsubsection{\(\mathcal{S}\times\mathcal{E}\)}\label{subsubsec:mixedS*E}
We can also consider \(\mathcal{S}\times\mathcal{E}\). 
The quotient surfaces $S$ constructed in Section \ref{subsubsec:S*E} also admit smooth projective
liftings to mixed characteristic.
For odd $p$, the curve and its automorphism lift over $W(k)[\zeta_p]$ as follows:
\[
 v^2=\frac{(1+(\zeta_p-1)u)^p-1}{(\zeta_p-1)^p},
 \qquad
 \beta(u,v)=(\zeta_pu+1,v),
\]
as in \cite[\S2]{Li20}.
Choose the canonical lift of the ordinary elliptic factor, together with the lift of its
nonzero \'etale $p$-torsion point. The resulting diagonal action remains free, and its
quotient gives a smooth projective lifting $\mathcal S$ of $S$.
In characteristic two, such a lifting of the order-four quotient is given
explicitly in \cite[\S1.2]{Lang95}.

Choose a lifting $\mathcal E$ of a supersingular elliptic curve over the same
base and set $\mathcal Z=\mathcal S\times\mathcal E$.
The generic surface $S_\eta$ has $q=1$, $p_g=0$, and $b_2=2$.
Consequently, the generic Hodge numbers and special Hodge--Witt numbers are
\[
 \begin{array}{c|cc}
 n&\bigl(h^{i,n-i}(Z_\eta)\bigr)_i
  &\bigl(h_W^{i,n-i}(Z_0)\bigr)_i\\ \hline
 2&(1,5,1)&(2,3,2)\\
 3&(0,4,4,0)&(1,3,3,1),
 \end{array}
\]
where the special-fiber calculation is given in Theorem~\ref{thm:strict-specialization}.

\subsubsection{A mixed-characteristic Godeaux--Serre example}
\label{subsubsec:mixed-godeaux-serre}

The approximation of $B\alpha_p$ used in \cite{LiYang} also yields a
strict mixed-characteristic polygon comparison. We construct a lifting
of such an approximation and take its product with a curve of genus two.
The resulting degree-three Hodge--Witt numbers are all nonnegative.

Let $k$ be an algebraically closed field of characteristic $p>0$. 
Put \( \mathcal O=W(k)[\pi]/(\pi^p-p)\) and 
\[f(T)=\frac{(1+\pi T)^p-1}{\pi^p}=T^p+\sum_{j=1}^{p-1}\frac{\binom pj}{p}\pi^jT^j.\]
The ring $\mathcal O$ is a complete discrete valuation ring with
uniformizer $\pi$. The polynomial $f$ is monic with coefficients in
$\mathcal O$, so \(\mathcal G=\operatorname{Spec}\bigl(\mathcal O[T]/(f(T))\bigr)
\) is finite flat of rank $p$. It is a commutative group scheme for the law
\(T_1*T_2=T_1+T_2+\pi T_1T_2.\)  Write
$K=\operatorname{Frac}(\mathcal O)$.  
Reduction modulo $\pi$ and the generic
coordinate $z=1+\pi T$ give
\begin{equation}\label{eq:gs-group-fibers}
 \mathcal G_k\simeq\alpha_p,
 \qquad \mathcal G_K\simeq\mu_p,
 \qquad \mathcal G_{\overline K}\simeq\mathbb Z/p.
\end{equation}
The relative Godeaux--Serre construction
\cite[Corollary~3.7]{Kothari} gives a complete intersection
$\mathcal Y\subset\mathbb P^N_{\mathcal O}$ of relative dimension four
with a free $\mathcal G$-action, whose quotient
$\mathcal X=\mathcal Y/\mathcal G$ is smooth projective over $\mathcal O$
with geometrically connected fibers. This result applies over any
noetherian local base, in particular over the mixed-characteristic ring
$\mathcal O$. If we take the ambient representation sufficiently large,
the equivariant Lefschetz theorem and projective bundle formula give
maps
\[
 X_t\longrightarrow B\mathcal G_t\times B\mathbb G_m
 \qquad(t=k,\overline K)
\]
which are Hodge $4$-equivalences; see
\cite[Propositions~5.10--5.11]{ABM} and \cite[\S4]{Kothari}.

Over $\overline K$, the group $\mathcal G_{\overline K}$ is finite
\'etale and linearly reductive. Thus its classifying stack has no
positive-degree Hodge cohomology. The factor $B\mathbb G_m$ contributes
the usual class in bidegree $(1,1)$, and hence
\begin{equation}\label{eq:gs-fourfold-generic}
 \begin{aligned}
 \bigl(h^{i,1-i}(X_{\overline K})\bigr)_{i=0}^{1}&=(0,0),\\
 \bigl(h^{i,2-i}(X_{\overline K})\bigr)_{i=0}^{2}&=(0,1,0),\\
 \bigl(h^{i,3-i}(X_{\overline K})\bigr)_{i=0}^{3}&=(0,0,0,0).
 \end{aligned}
\end{equation}
On the special fiber, the Hodge $4$-equivalence is also a
de Rham--Witt $4$-equivalence by \cite[Proposition~4.12]{LiYang}.
The calculation of $B\alpha_p\times B\mathbb G_m$ in
\cite[Theorem~5.11 and Corollary~5.12]{LiYang} therefore gives
\begin{equation}\label{eq:gs-fourfold-special}
 \begin{aligned}
 \bigl(h_W^{i,1-i}(X_k)\bigr)_{i=0}^{1}&=(0,0),\\
 \bigl(h_W^{i,2-i}(X_k)\bigr)_{i=0}^{2}&=(0,1,0),\\
 \bigl(h_W^{i,3-i}(X_k)\bigr)_{i=0}^{3}&=(1,-2,1,0).
 \end{aligned}
\end{equation}

Choose an ordinary smooth projective curve $C_0/k$ of genus two, and a
smooth projective lifting $\mathcal C/\mathcal O$. Such ordinary curves
exist in every characteristic, and the lifting exists by smoothness of
the moduli stack of smooth curves \cite[Tag~0E9C]{Stacks}.
Set
\[
 \mathcal Z=\mathcal X\times_{\mathcal O}\mathcal C.
\]
This is smooth projective of relative dimension five with geometrically
connected fibers. Since $C_0$ is ordinary, its Hodge--Witt cohomology is
finite free over $W(k)$, with the same ranks as its Hodge cohomology.
The cohomological K\"unneth formula
\cite[III, Corollary~2.1.2]{EkedahlII} has no Tor terms and yields
\begin{equation}\label{eq:gs-ordinary-curve-product}
 h_W^{i,j}(Z_k)
 =h_W^{i,j}(X_k)+2h_W^{i-1,j}(X_k)
  +2h_W^{i,j-1}(X_k)+h_W^{i-1,j-1}(X_k),
\end{equation}
where entries with a negative index are zero. Substituting these values into the K\"unneth formula gives
\begin{equation}\label{eq:gs-mixed-numbers}
 \begin{aligned}
 \bigl(h^{i,2-i}(Z_{\overline K})\bigr)_{i=0}^{2}
   &=(0,2,0)
    =\bigl(h_W^{i,2-i}(Z_k)\bigr)_{i=0}^{2},\\
 \bigl(h^{i,3-i}(Z_{\overline K})\bigr)_{i=0}^{3}
   &=(0,2,2,0),\\
 \bigl(h_W^{i,3-i}(Z_k)\bigr)_{i=0}^{3}
   &=(1,0,3,0).
 \end{aligned}
\end{equation}

\subsubsection{Li's construction with arbitrary exponent}
\label{subsubsec:li}

Let $p\geqslant5$ and put $g=(p-1)/2$. In this subsection we consider a variation of the construction
in \cite[\S2]{Li20}. Write $\pi=\zeta_p-1$. Over a suitable finite extension
$\mathcal O_K$ of $\mathbb Z_p[\zeta_p]$, let $\mathcal C$ be the smooth proper
curve with affine equation
\[
 v^2=\frac{(1+\pi u)^p-1}{\pi^p},
 \qquad \beta(u,v)=(\zeta_pu+1,v).
\]
Choose $\mathcal E/\mathcal O_K$ to be the base change of the canonical lift of an ordinary
elliptic curve, with a section $P$ lifting a nonzero \'etale $p$-torsion point.
Write $\tau_P$ for translation by $P$.
For $1\leqslant m\leqslant p-1$, set
\[
 \mathcal X_m=
 \frac{\mathcal C\times\mathcal C\times\mathcal E}
      {\langle(\beta,\beta^m,\tau_P)\rangle}.
\]
These actions are free, so the quotients are smooth and projective.

All the geometric special fibers are isomorphic. Indeed, over an algebraically
closed residue field $k$, the special fiber of $\mathcal C$ is
\[
 C_0:\ v^2=u^p-u,\qquad \beta(u,v)=(u+1,v).
\]
For $m\in\mathbb F_p^\times$, the automorphism
\[
 \tau_m(u,v)=(mu,\sqrt m\,v)
\]
satisfies $\tau_m\beta\tau_m^{-1}=\beta^m$. Thus
$(\mathrm{id},\tau_m,\mathrm{id})$ induces an isomorphism between the special
fibers of $\mathcal X_1$ and $\mathcal X_m$. A single residue-field extension
containing $\mathbb F_{p^2}$ suffices for all these isomorphisms. We write $Z$
for the common geometric special fiber and $X_{m,\eta}$ for the generic fiber.

For $m\in\mathbb F_p^\times$, put \begin{equation}
  \label{eq:li-count} c_m=\#\{(a,b):1\leqslant a,b\leqslant g, a+mb\equiv0\pmod p\}.
\end{equation}
If $\langle x\rangle_p$ denotes the least positive residue of a nonzero class
modulo $p$, then
\begin{equation}\label{eq:li-count-floor}
 c_m
 =\#\{1\leqslant b\leqslant g:\langle mb\rangle_p>g\}
 =\sum_{b=1}^{g}
   \left(\left\lfloor\frac{mb+g}{p}\right\rfloor
         -\left\lfloor\frac{mb}{p}\right\rfloor\right).
\end{equation}
Here the last expression uses the representative $1\leqslant m\leqslant p-1$.

\begin{proposition}\label{prop:li-all-hodge}
For every $1\leqslant m\leqslant p-1$, the generic-fiber Hodge numbers are
\begin{equation}\label{eq:li-all-hodge}
 \begin{aligned}
 \bigl(h^{i,2-i}(X_{m,\eta})\bigr)_{i=0}^{2}
   &=(c_m,p+2-2c_m,c_m),\\
 \bigl(h^{i,3-i}(X_{m,\eta})\bigr)_{i=0}^{3}
   &=(c_m,p+1-c_m,p+1-c_m,c_m).
 \end{aligned}
\end{equation}
\end{proposition}

\begin{proof}
Let $\chi(\beta)=\zeta_p$. By \cite[Proposition~2.1]{Li20} and Serre duality,
\[
 \mathrm H^0(C_\eta,\Omega^1_{C_\eta})
   =\bigoplus_{a=1}^{g}\chi^a,
 \qquad
 \mathrm H^1(C_\eta,\mathcal O_{C_\eta})
   =\bigoplus_{a=1}^{g}\chi^{-a}.
\]
On the second curve factor, the action of $\beta^m$ multiplies these character
exponents by $m$. Thus the invariant parts of bidegrees $(2,0)$ and $(0,2)$
on $C_\eta\times C_\eta$ both have dimension $c_m$. Each of the two mixed
summands in bidegree $(1,1)$ has invariant dimension $g-c_m$. Including the
two curve classes gives invariant dimension
\(2+2(g-c_m)=p+1-2c_m\)
in bidegree $(1,1)$. The odd-degree Hodge cohomology of the two-curve product
has no invariants. Translation acts trivially on the Hodge cohomology of
$E_\eta$. The K\"unneth formula, followed by taking invariants in
characteristic zero, now gives \eqref{eq:li-all-hodge}.
\end{proof}

The integers $c_m$ satisfy \(c_{m^{-1}}=c_m\) and \(c_{-m}=g-c_m.\) Indeed, 
interchanging $a$ and $b$ in \eqref{eq:li-count} proves the first symmetry.
For each $b$, exactly one of $\langle mb\rangle_p$ and
$\langle-mb\rangle_p$ lies in $\{g+1,\ldots,p-1\}$, proving the second.
Moreover, $c_m=0$ if and only if $m=1$, and $c_m=g$ if and only if $m=-1$.
Indeed, put $A=\{1,\ldots,g\}\subset\mathbb F_p^\times$. If $c_m=g$, then $-mA=A$.
Summing the elements of these sets gives
\((-m-1)\sum_{a\in A}a=0\) in $\mathbb F_p$, while
$\sum_{a\in A}a=g(g+1)/2$ is nonzero modulo $p$.
Hence $m=-1$. The converse is immediate, and the characterization of $c_m=0$
follows from $c_{-m}=g-c_m$. 

Now we compute the Hodge--Witt numbers of the special fiber. 
First consider the curve \(C_0\). Its Hasse--Witt matrix is obtained from the coefficients of
$(u^p-u)^g$, and the relevant coefficients all vanish.
Thus the slopes of $\mathrm H^1_{\mathrm{cris}}(C_0/W)[1/p]$ are all $1/2$.
Rational crystalline descent for the free quotient gives
\begin{equation}\label{eq:li-newton-low}
   \begin{array}{c|c}
 n&\text{Newton slopes of }Z\\ \hline
 1&0,\ 1\\
 2&1\text{ with multiplicity }p+2\\
 3&1,\ 2\text{, each with multiplicity }p+1.
 \end{array}
\end{equation}

On the special fiber, the representation
$\mathrm H^0(C_0,\Omega^1_{C_0})$ is a single unipotent Jordan block $J_g$:
in the basis $u^a\,du/v$, $0\leqslant a<g$, the action is translation in $u$.
Since $J_g^\vee\simeq J_g$, descent of top forms gives
\( h^{0,3}(Z)=h^{3,0}(Z) =\dim_k(J_g\otimes_kJ_g)^{\mathbb Z/p} =\dim_k\operatorname{End}_{\mathbb Z/p}(J_g)=g.\)
Ekedahl's inequality \cite[IV, Theorem~3.3]{Ekedahl} gives
$h_W^{0,3}(Z)\leqslant g$.
On the other hand, Theorem~\ref{thm:main}, applied to $\mathcal X_{p-1}$
in degree three at cutoff one, gives $h_W^{0,3}(Z)\geqslant g$.
Hence $h_W^{0,3}(Z)=g$.
Furthermore, $\chi(\mathcal O_Z)=0$ and
$h_W^{0,0}(Z)=h_W^{0,1}(Z)=1$.
Crew's formula \cite[Theorem~4]{Crew} consequently gives
$h_W^{0,2}(Z)=h_W^{0,3}(Z)=g$.
Since the degree-two and degree-three Newton polygons have no slopes below one,
we obtain \(T^{0,2}(Z)=T^{0,3}(Z)=g.\)
Domino duality gives $T^{1,3}(Z)=T^{0,2}(Z)$ and
$T^{1,2}(Z)=T^{0,3}(Z)$
\cite[IV, Corollary~3.5.1]{EkedahlI}.
Substituting these domino numbers and slope numbers into the definition gives the Hodge--Witt numbers of $Z$:
\begin{equation}\label{eq:li-special-hodge-witt}
 \bigl(h_W^{i,2-i}(Z)\bigr)_i=(g,3,g),\qquad
 \bigl(h_W^{i,3-i}(Z)\bigr)_i=(g,g+2,g+2,g).
\end{equation} 

\begin{remark}\label{rem:li-count-values}
The integers $c_m$ need not attain every value between $0$ and $g$, and their
intermediate values need not occur twice. For example, when $p=11$, $1$ and $4$ do not occur, 
whereas $2$ and $3$ each occur four times. 
\end{remark}

Notice that when $p=5$ or $7$, all possible generic Hodge polygons are realized by one of the \(\mathcal{X}_m\).
Let $X_\eta$ be the generic fiber of any smooth proper lifting of $Z$ covered by
Theorem~\ref{thm:main}. Crystalline comparison and generic Hodge--de Rham degeneration give
\(b_1=2,b_2=p+2,b_3=2p+2.\)
The Newton slopes in \eqref{eq:li-newton-low}, together with
$t_{H,n}=t_{N,n}$, give \(h^{1,0}=h^{0,1}=1\), \(h^{1,1}+2h^{2,0}=p+2
=h^{0,2}+h^{1,1}+h^{2,0}.\)
Thus $h^{2,0}=h^{0,2}$. Put $t=h^{0,2}(X_\eta)$.
Since coherent Euler characteristics are preserved under specialization,
\( 0=\chi(\mathcal O_Z)=\chi(\mathcal O_{X_\eta}) =1-1+t-h^{0,3}(X_\eta),\)
so $h^{0,3}(X_\eta)=t$. Serre duality in dimension three gives
\( h^{3,0}=h^{0,3}\) and \(h^{2,1}=h^{1,2},\)
and thus the Hodge numbers must be \[
 \begin{aligned}
 \bigl(h^{i,2-i}(X_\eta)\bigr)_{i=0}^{2}
   &=(t,p+2-2t,t),\\
 \bigl(h^{i,3-i}(X_\eta)\bigr)_{i=0}^{3}
   &=(t,p+1-t,p+1-t,t).
 \end{aligned}
\]
All remaining Hodge numbers are determined by Serre duality and $h^{0,0}=1$.
Theorem~\ref{thm:main} gives
\[
 0\leqslant t=h^{0,3}(X_\eta)\leqslant h_W^{0,3}(Z)=g.
\]
For $p=5$ and $p=7$, respectively, the values of $(c_1,\ldots,c_{p-1})$ are
$(0,1,1,2)$ and $(0,2,1,2,1,3)$. Thus every permitted value of $t$ is realized by one of
the liftings $\mathcal X_m$.

For all sufficiently large primes, the values of $c_m$ do not cover all integers in $[0,g]$.
As $p$ tends to $\infty$ through the primes, for each fixed positive odd integer $k$, one has
$c_k=g/2-g/(2k)+O_k(1)$ and $c_{-k}=g/2+g/(2k)+O_k(1)$,
whereas for each fixed positive even integer $k$, one has $c_k=g/2+O_k(1)$.
Indeed, the floor formula counts lattice points in a union of $O(k)$ intervals; replacing
these counts by interval lengths gives an error $O_k(1)$, and $c_{-k}=g-c_k$.
Allowing $m$ to vary with $p$, it can also be shown that the set of possible limits of $c_m/g$
is exactly
\begin{equation}\label{eq:li-limit-set}
 \left\{\frac12\right\}
 \cup\left\{\frac12\pm\frac1{2k}:k\geqslant1\text{ odd}\right\}.
\end{equation}
In particular, uniformly for $m\ne\pm1$, one has
$g/3-o(g)\leqslant c_m\leqslant2g/3+o(g)$.

\section*{Acknowledgements}
We are very grateful to Shizhang Li for his insightful suggestions and stimulating discussions.
Y.Y.\ is partially supported by a grant from the National Natural Science Foundation of China
(NSFC-12231001).

\section*{AI provenance}
OpenAI language models, including ChatGPT, were used in developing arguments and preparing the
exposition of this manuscript. The author is responsible for the mathematical content and its final
verification.

\end{document}